\documentclass[12pt,a4paper]{amsart}
\usepackage{fullpage,color}
\usepackage[mac]{inputenc}
\usepackage{amsmath,amsfonts,amssymb,amsthm,amscd}
\usepackage{graphicx}
\usepackage{verbatim}

\usepackage{hyperref}
\usepackage{mathtools}
\usepackage{tikz}

\newtheorem{theorem}{Theorem}

\newtheorem{lemma}[theorem]{Lemma}
\newtheorem{corollary}[theorem]{Corollary}

\title{Explicit bound for the number of primes in arithmetic progressions assuming the generalized Riemann hypothesis}
\author{Anne-Maria Ernvall-Hyt\"onen}
\address{Anne-Maria Ernvall-Hyt\"onen \\
Mathematics and Statistics, P. O 68, 00014 University of Helsinki, Finland}\thanks{The work of Ernvall-Hyt\"onen was supported by the Emil Aaltonen foundation.}
\email{anne-maria.ernvall-hytonen@helsinki.fi}
\author{Neea Paloj\"arvi}
\address{Neea Paloj\"arvi \\
Matematik och Statistik, {\AA}bo Akademi University, Domkyrkotorget 1, 20500 {\AA}bo, Finland}
\curraddr{Mathematics and Statistics, P. O 68, 00014 University of Helsinki, Finland}
\email{neea.palojarvi@helsinki.fi}
\begin{document}
\maketitle

\begin{abstract}
We prove an explicit error term for the $\psi(x,\chi)$ function assuming the Generalized Riemann Hypothesis. Using this estimate, we prove a conditional explicit bound for the number of primes in arithmetic progressions. 
\end{abstract}

\maketitle

\section{Introduction}
Explicit estimates for the distribution of primes assuming the Riemann Hypothesis has been widely investigated. A good starting point is Schoenfeld's thorough article \cite{schoenfeld} which is the second part of a similarly impressive paper by Rosser and Schoenfeld \cite{rosserschoenfeld}. B\"uthe has proved explicit results assuming the partial Riemann Hypothesis, namely, that the Riemann Hypothesis holds up to some height \cite{buthe}. The aim of this article is to treat primes in arithmetic progressions assuming the Generalized Riemann Hypothesis.

Let now 
\begin{equation*}
\pi(x;q,a)=\sum_{\substack{p\leq x\\ p\equiv a\bmod q}}1,
\end{equation*}
where $\gcd(a, q) = 1$, compute the number of primes up to $x$ which are congruent to $a$ modulo $q$. By de la Vall\'{e}e Poussin \cite{poussin}
\begin{equation*}
    \pi(x; q, a) \sim \frac{x}{\varphi(q)\log{x}}.
\end{equation*}

Furthermore, we can also consider the term $\pi(x; q, a)$ assuming the \textit{Generalized Riemann Hypothesis} (GRH), indeed that if $L(s,\chi)=0$, where $L$ is a Dirichlet $L$-function and $s$ is not a negative real number, then $\Re s=1/2$. It is known (see e.g. \cite[Chapters 19 and 20]{davenport}) that assuming the GRH, we have
\begin{equation}
\label{piLi}
	\pi(x;q,a)=\frac{\mathrm{Li}(x)}{\varphi(q)}+O\left(\sqrt{x}\log{x}\right)
\end{equation}
where $\mathrm{Li}(x)=\int_2^x \frac{dt}{\log{t}}\sim\frac{x}{\log{x}}$. Thus it is reasonable to prove explicit bounds of this sizes. Recently Bennett, Martin, O'Bryant and Rechnitzer \cite[Theorem 1.3]{theta} gave examples of various constants $c_\pi(q)$ and $x_\pi(q)$ such that
\begin{equation*}
\left|\pi(x;q,a)-\frac{\mathrm{Li}(x)}{\varphi(q)}\right|<c_\pi(q)\frac{x}{\log^2 x} \quad\text{for all } x\geq x_\pi(q).
\end{equation*}

Explicit bounds for functions are important in computing. However, the motivation for writing this article originally came from the first named author's collaboration \cite{EMS:Euler} where $p$-adic evaluations of Euler's divergent series where investigated in arithmetic sequences. In Theorem 5, assuming the GRH, the author proved that there is a bound such that if certain constant is below this bound, then there is a $p$ such that the $p$-adic evaluation of the series is not rational or there are $p$ and $q$ such that the $p$ and $q$-adic evaluations are not equal. However, the authors were not able to give an explicit bound for this constant due to inexplicit formulation of the error term in the bound for the number of primes in an arithmetic progression assuming the GRH. The results in this paper could be used to sharpen Lemma 1 and subsequently the proof of Theorem 5 in that paper, and thereby, to derive such an explicit constant.

In this paper, where 
\begin{equation*}\mathrm{li}(x)=\lim_{\epsilon\to 0+}\int_0^{1-\epsilon}\frac{dt}{\log{t}}+\int_{1+\epsilon}^x\frac{dt}{\log{t}},
\end{equation*}  we obtain the following bound on the number of primes in an arithmetic progression:

\begin{theorem}
\label{pi}
Assume the GRH. Let $a$ and $q$ be integers and $x$ be a real number such that $\gcd(a, q) = 1$, $q\geq 3$ and $x\geq q$. Then we have
\begin{multline*}
\left|\pi(x;q,a)-\frac{\mathrm{li}(x)}{\varphi(q)}\right| <  \left(\frac{1}{8\pi \varphi(q)}+\frac{1}{6\pi}\right)\sqrt{x}\log x+\left(0.184\log q+8.396\right)\sqrt{x} \\
+\left(6.05\log q+158.745\right)\frac{\sqrt{x}}{\log x}+ \left(5.048\log^2 q+152.085\log{q}+1731.270\right)\frac{\sqrt{x}}{\log^2 x}\\
+\left(0.184\log q+8.250\right)x^{1/4}\log\log{x}+\left(5.254\log^2 q+121.765\log{q}+939.260\right)x^{1/4} \\
+\left(80.768\log^2 q+1753.168\log{q}+11605.056\right)\frac{\sqrt{x}}{\log^3 x}-237.934.
\end{multline*}
\end{theorem}

Further, this can be simplified in the following way.
\begin{corollary}
\label{piCorollary}
Assume the GRH. Let $a$ and $q$ be integers and $x$ be a real number such that $\gcd(a, q) = 1$, $q\geq 3$ and $x\geq q$. Then we have
\begin{multline*}
\left|\pi(x;q,a)-\frac{\mathrm{li}(x)}{\varphi(q)}\right| \leq \left(\frac{1}{8\pi \varphi(q)}+\frac{1}{6\pi}\right)\sqrt{x}\log x \\
+\left(0.184\log q+12969.946\right)\sqrt{x}-237.934.
\end{multline*}
\end{corollary}

In order to prove the claim, we apply Theorem \ref{psi} described below. Then, the proof of Theorem \ref{pi} is a straightforward partial summation argument. It would be possible to obtain a sharper bound but since the first term on the right side is clearly the dominating for large values of $x$, and the full result is long and complicated, we decided to state the result in this form which easily reveals the order of magnitude. 

The proof of Theorem \ref{psi} is more involved. In the theorem, we consider the function
\begin{equation*}
\psi(x;q,a)=\sum_{\substack{n \le x \\ n \equiv a \pmod q}} \Lambda(n) = \sum_{\substack{p^\alpha \le x \\ p^\alpha \equiv a \pmod q}} \log p
\end{equation*}
where $\gcd(a, q) = 1$ and $\Lambda(n)$ is the von Mangoldt function. We will first express the function $\psi(x;q,a)$ in terms of the $\psi$-functions involving multiplicative characters, which we can then write in terms of the von Mangoldt's formula, and then carefully bound the terms. To emphasise that the following result does not require the full GRH but the hypothesis restricted to certain moduli, we use the notation GRH$(q)$ to mean that the GRH holds for Dirichlet characters modulo $q$. 

\begin{theorem}
\label{psi} 
Let $x\geq 2$ be a real number and $q\geq 3$ and $a$, where $\gcd(a, q) = 1$, be integers.  Assume that the GRH$(q)$ holds for the number $q$ and all moduli dividing $q$. We have
\begin{multline*}
\left|\psi(x;q,a)-\frac{x}{\varphi(q)}\right|< \left(\frac{1}{8\pi \varphi(q)}+\frac{1}{6\pi}\right)\sqrt{x}\log^2 x+ \left(0.184\log q+8.250\right)\sqrt{x}\log x\\
+(5.314\log q+124.318)\sqrt{x}+(5.048\log^2q+109.573\log q+725.316)\frac{\sqrt{x}}{\log {x}}\\ 
 +\left(2.015\log q+0.5\right)\log x +R_1(q),
\end{multline*}
where the term $R_1(q)$ describes the contribution coming from the terms which are asymptotically at most size $O\left(1\right)$ with respect to $x$. The term $R_1(q)$ can be written as
\begin{align*}
&R_1(q)=0.014\sqrt{q}\log{q}+0.034\sqrt{q}+3.679\log q+263.886
\end{align*}
if $3\leq q<4\cdot 10^5$,
\begin{align*}
&R_1(q)=1.858\log^2{q}+3.679\log q+104.626
\end{align*}
if $4\cdot 10^5\leq q<10^{29}$ and
\begin{align*}
&R_1(q)=(0.297\log\log q+0.603)\log^2{q}+3.679\log q+104.626
\end{align*}
otherwise.
\end{theorem}

\section{Preliminary inequalities}
\label{secUsefulInequalities}

Before moving to actual number theoretical proofs, we prove a couple of estimates which we are going to use in Sections \ref{secProofPsi} and \ref{secProofPi}. They are written in the beginning so that they do not interrupt the flow of the number theoretical proofs.

\subsection{Estimates related to the function $\psi(x;q,a)$}
In Section \ref{secProofPsi} we prove a quite long estimate for a function which is closely related to the function $\psi(x;q,a)$ (see Theorem \ref{largeRho}) and later we get some very long estimates in the proof of Theorem \ref{psi}. We would like to simplify these expressions, and hence the following six estimates are proved.

We start with terms which are asymptotically of size $O(x^{0.423}\log{x})$ in the end of the proof of Theorem \ref{largeRho}.

\begin{lemma}
\label{lemma423log}
Assume $x \geq 2$ and $T=x^{0.577}+8.509$. Then
\begin{equation*}
12.294\frac{x\log x}{T+1}+\frac{2xe}{\pi(T-1)}\log{x}+\frac{2.385xe}{\pi (T-1)\log{x}}\log^2{\left(T+1\right)}\leq 14.712x^{0.423}\log{x}.
\end{equation*}
\end{lemma}
\begin{proof}
Let us set $T=x^{0.577}+8.509$ and the left hand side of the wanted inequality is
\begin{multline}
\label{eqxIn423log}
=\frac{12.294x\log x}{x^{0.577}+9.509}+\frac{2xe}{\pi(x^{0.577}+7.509)}\log{x} \\
+\frac{2.385xe}{\pi (x^{0.577}+7.509)\log{x}}\log^2{\left(x^{0.577}+9.509\right)}.
\end{multline}
We divide the previous one by $x^{0.423}\log{x}$. Taking a look to the derivatives of these new three terms, the first two terms are increasing, but easily bounded, for $x>0$. 

Let us now analyse the third term. It is 
\begin{equation*}
\frac{2.385x^{0.577}e}{\pi (x^{0.577}+7.509)\log^2{x}}\log^2{\left(x^{0.577}+9.509\right)},
\end{equation*}
and we change the variable: $x^{0.577}=y$. 
We get 
\begin{equation*}
\frac{2.385\cdot 0.577^2ye}{\pi (y+7.509)\log^2{y}}\log^2{\left(y+9.509\right)}.
\end{equation*} 
Let us ignore the the constants for a while, and investigate when 
\begin{equation}
\label{ineqY}
\frac{y\log^2{\left(y+9.509\right)}}{(y+7.509)\log^2{y}}\leq 1.
\end{equation}
This is equivalent to considering when $y\log^2{\left(y+9.509\right)}\leq (y+7.509)\log^2{y}$. We look at the difference:
\begin{align*}
&y\log^2{\left(y+9.509\right)}-(y+7.509)\log^2{y}\\ 
&\quad =y(\log(y+9.509)+\log(y))(\log(y+9.509)-\log(y))-7.509\log^2y\\ 
&\quad \leq y(\log(y+9.509)+\log(y))\cdot \frac{9.509}{y}-7.509\log^2y \\ 
&\quad =9.509\log(y^2+9.509y)-7.509\log^2y.
\end{align*}
Denote $g(y)=9.509\log(y^2+9.509y)-7.509\log^2y$. Then
\begin{equation*}
g'(y)=\frac{90.4211 + 19.018 y + (-142.806 - 15.018 y) \log(y)}{y (9.509 + y)}\leq 0
\end{equation*}
when $y\geq 2.2$ since the numerator can be written as
\begin{align*}
&90.4211 + 19.018 y + (-142.806 - 15.018 y) \log y \\
&\quad=\left(90.4211-142.806 \log y \right)+\left(19.018- 15.018\log y\right)y,
\end{align*}
and is thus decreasing, and the numerator is negative at $y=2.2$.
Notice now that $g(16)\approx-0.56<0$ and thus inequality \eqref{ineqY} is satisfied for $y\geq 16$. Then the whole expression which is obtained when \eqref{eqxIn423log} is divided by $x^{0.423}\log{x}$ can be bounded as
\begin{align*}
&\frac{12.294x^{0.577}}{x^{0.577}+9.509}+\frac{2x^{0.577}e}{\pi(x^{0.577}+7.509)}+\frac{2.385x^{0.577}e}{\pi (x^{0.577}+7.509)\log^2{x}}\log^2{\left(x^{0.577}+9.509\right)}\\ 
& \quad \leq 12.294+\frac{2e}{\pi}+\frac{2.385e\cdot 0.577^2}{\pi}<14.712.
\end{align*}

Let us now move to the case with $y\leq 16$. Assume first $y\geq 2.2$. Now the first term, obtained when \eqref{eqxIn423log} is divided by $x^{0.423}\log{x}$, is at most
\begin{equation*}
\frac{12.294y}{y+9.509}<7.712.
\end{equation*}
The second term is at most 
\begin{equation*}
\frac{2ey}{\pi(y+7.509)}< 1.2.
\end{equation*}
The third term is at most
\begin{align*}
&\frac{2.385\cdot 0.577^2ye}{\pi (y+7.509)\log^2{y}}\log^2{\left(y+9.509\right)} \\
&\quad\leq \frac{2.385\cdot 0.577^2\cdot 16 e}{\pi (16+7.509)\log^2{2.2}}\log^2{\left(2.2+9.509\right)}< 4.554.
\end{align*}
The sum of these terms is smaller than $13.466$. When $2^{0.577}\leq y\leq 2.2$, the first term is at most $\frac{12.294y}{y+9.509}<2.310$, the second term is smaller than $0.393$ and the third term is at most
\begin{equation*}
\frac{2.385\cdot 0.577^2\cdot 2.2 e}{\pi (2.2+7.509)\log^2{2^{0.577}}}\log^2{\left(2.2+9.509\right)}< 5.892.
\end{equation*}
The sum of the terms is smaller than $10.281$.

The proof of this lemma is complete.
\end{proof}

In next three lemmas we estimate other terms which appear in the proof of Theorem \ref{largeRho} and which do not go to zero when $x$ goes to infinity.

\begin{lemma}
\label{lemma0423loglog}
Assume $x \geq 2$, $q \geq 3$ and $T=x^{0.577}+8.509$. Then
\begin{equation*}
\frac{17.472xe}{\pi (T-1)\log{x}}\log{\left(T+1\right)}\log{\log{(q(T+1))}}<18.610x^{0.423}\log\log{(qx)}.
\end{equation*}
\end{lemma}
\begin{proof}
Dividing the left hand side of the inequality by $x^{0.423}\log\log{(qx)}$ and setting $T=x^{0.577}+8.509$ we get
\begin{equation}
\label{eqx0423loglog}
\frac{17.472x^{0.577}e}{\pi (x^{0.577}+7.509)\log{x}\cdot\log{\log{(qx)}}}\log{\left(x^{0.577}+9.509\right)}\log{\log{(q(x^{0.577}+9.509))}}.
\end{equation}
The function $x^{0.577}/(x^{0.577}+7.509)$ is increasing but always less than $1$ and the functions 
\begin{equation*}
\log{\left(x^{0.577}+9.509\right)}/\log{x}\quad \text{and} \quad\log{\log{(q(x^{0.577}+9.509))}}/ \log{\log{(qx)}}
\end{equation*}
are decreasing for $x \geq 2$, $q \geq 3$. 
Thus, let us divide the consideration to different cases:

First, we can set $q=3$ in formula \eqref{eqx0423loglog}. Let us then consider the following cases 
\begin{equation}
\label{intervals}
\begin{aligned}
x\in &[2, 2.001), \quad [2.001, 2.005), \quad  [2.005, 2.02), \quad  [2.02, 2.1), \\
& [2.1, 2.5), \quad  [2.5, 5.3), \quad  [5.3, 85)  \quad  [85, \infty).
\end{aligned}
\end{equation}
We can use the left endpoints of the intervals in the expressions
\begin{equation*}
\log{\left(x^{0.577}+9.509\right)}/\log{x} \quad\text{and}\quad \log{\log{(3(x^{0.577}+9.509))}}/ \log{\log{(3x)}}
\end{equation*} 
and the right endpoints (including a limit going to infinity) of the intervals to the expression
\begin{equation*}
x^{0.577}/(x^{0.577}+7.509).
\end{equation*} 
Then formula \eqref{eqx0423loglog} is smaller than $18.610$, $18.606$, $18.586$, $18.610$, $18.525$, $18.492$, $18.018$ and $8.896$ where the upper bounds correspond to the intervals in \eqref{intervals}.
Thus the wanted upper bound is found.
\end{proof}

\begin{lemma}
\label{lemma0423}
Assume $x \geq 2$, $q \geq 3$ and $T=x^{0.577}+8.509$. Then
\begin{multline*}
\left(\frac{7.032}{T+1}+\frac{2e\gamma}{\pi(T-1)}\right)x+\frac{xe}{\pi (T-1)\log{x}}\left(4.77\log{q}\log{(T+1)} \vphantom{\frac{13}{8}\log{\left(\frac{T^2}{4}+\frac{T}{2}+\frac{5}{2}\right)}}\right.\\
 \left.-3.276\log{\left(T+1\right)} +\frac{13}{8}\log{\left(\frac{T^2}{4}+\frac{T}{2}+\frac{5}{2}\right)} \right) \\
 <(2.382\log q+8.018)x^{0.423}+4.35825\frac{ex^{0.423}}{\pi\log {x}}.
\end{multline*}
\end{lemma}
\begin{proof}
First we notice that we have 
\begin{equation*}
\frac{\log\left(\frac{T^2}{4}+\frac{T}{2}+\frac{5}{2}\right)}{T-1}< \frac{1.154\log x+2.682}{x^{0.577}},
\end{equation*} 
when $x\geq 2$,  since
\begin{align*}
\frac{1.154\log x+2.682}{\log\left(\frac{T^2}{4}+\frac{T}{2}+\frac{5}{2}\right)}& =\frac{\log{\left(e^{2.682}x^{1.154}\right)}}{\log\left(x^{1.154}/4 + 4.7545 x^{0.577} + 24.85527025\right)} \\
&>\frac{\log{\left(e^{2.682}x^{1.154}\right)}}{\log\left(\left(1/4+3.188+11.170\right)x^{1.154}\right)} \\
&=\frac{\log{\left(e^{2.682}x^{1.154}\right)}}{\log\left(14.608x^{1.154}\right)}>1 \\
& >\frac{x^{0.577}}{x^{0.577}+7.509}.
\end{align*}

Now we show that 
\begin{equation*}
\frac{\log (T+1)}{T-1}\leq \frac{0.577\log x}{x^{0.577}}.
\end{equation*} 
Change the variable: $x^{0.577}=y$. Now we need to prove this inequality for $y\geq 2^{0.577}$. This inequality is equivalent to showing that the function
\begin{equation*}
f(y)=y\log(y+9.509)-(y+7.509)\log y
\end{equation*}
is negative. We have
\begin{equation*}
f'(y)=\frac{y}{y+9.509}-\frac{7.509}{y}-\log y+\log(y+9.509)-1
\end{equation*}
and
\begin{equation*}
f''(y)=\frac{678.972 + 52.3851 y + 7.509 y^2}{y^2 (9.509 + y)^2}>0
\end{equation*}
when $y>0$, and hence the function $f'(y)$ is increasing. Since $\lim_{y\rightarrow \infty}f'(y)=0$, so the function $f(y)$ is decreasing. Hence, it suffices to check whether the inequality is true at the beginning and it is since we have $f(2^{0.577})\approx -0.02$.

Thus we can estimate:
\begin{align*}
&\left(\frac{7.032}{T+1}+\frac{2e\gamma}{\pi(T-1)}\right)x \\
&\quad+\frac{\left(4.77\log{q}-3.276\right)\log{\left(T+1\right)}+1.625\log{\left(\frac{T^2}{4}+\frac{T}{2}+\frac{5}{2}\right)}}{\pi (T-1)\log{x}}xe\\
& \quad< \left(\frac{7.032}{x^{0.577}}+\frac{2e\gamma}{\pi x^{0.577}}\right)x+\frac{4.77\log{q}-3.276}{\pi \log x}\cdot \frac{0.577\log x}{x^{0.577}}xe \\
&\quad\quad+\frac{1.625}{\pi \log x}\cdot\frac{2.682+1.154\log x}{x^{0.577}}xe\\ 
& \quad< (2.382\log q+8.018)x^{0.423}+4.35825\frac{ex^{0.423}}{\pi\log {x}}.
\end{align*}
\end{proof}

\begin{lemma}
\label{lemma0432loglogDivlog}
Assume $x \geq 2$, $q \geq 3$ and $T=x^{0.577}+8.509$. Then
\begin{align*}
& \frac{xe}{\pi (T-1)\log{x}}\left(32\left(\log{\log{(q(T+1))}}\right)^2+(17.472\log{q}-12)\log{\log{(q(T+1))}}\right) \\
&\quad<127.562\frac{x^{0.423}\left(\log\log{(qx)}\right)^2}{\log{x}}+\left(32.449\log{q}-1.720\right)\frac{x^{0.423}\log\log{(qx)}}{\log{x}}.
\end{align*}
\end{lemma}
\begin{proof}
Let us first prove that the first term on the left hand side is at most the first term on the right hand side. Then we prove that the second term on the left hand side is at most the second term on the right hand side. This case is divided into two parts: first we derive the coefficient $32.449\log{q}$ and then the coefficient $-1.720$.

Let us now start by proving the inequality
\begin{equation*}
\frac{xe}{\pi (T-1)\log{x}}32\left(\log{\log{(q(T+1))}}\right)^2<127.562\frac{x^{0.423}\left(\log\log{(qx)}\right)^2}{\log{x}}.
\end{equation*}
We can divide the term on the left hand side of the inequality by 
\begin{equation*}
\frac{x^{0.423}\left(\log\log{(qx)}\right)^2}{\log{x}}.
\end{equation*} 
Then the term is a product of functions 
\begin{equation*}
\frac{32x^{0.577}e}{\pi (x^{0.577}+7.509)} \qquad \text{and} \qquad \frac{\left(\log{\log{(q(T+1))}}\right)^2}{\left(\log{\log{(qx)}}\right)^2}.
\end{equation*} 
The first one is increasing but always at most $32e/\pi$ and the second one is decreasing for all $x \geq 2$, $q \geq 3$. Thus the maximum value of the first term on the left hand side of the inequality under the consideration is smaller than
\begin{equation*}
\frac{32e}{\pi }\frac{\left(\log{\log{(q(T+1))}}\right)^2}{\left(\log{\log{(qx)}}\right)^2}<127.562.
\end{equation*} 

Let us now move on to the second term, that is to proving the inequality
\begin{equation*}
\frac{xe(17.472\log{q}-12)}{\pi (T-1)\log{x}}\log{\log{(q(T+1))}}<\left(32.449\log{q}-1.720\right)\frac{x^{0.423}\log\log{(qx)}}{\log{x}}.
\end{equation*}
First, we divide the term on the left hand side by $x^{0.423}\log\log{(qx)}/\log{x}$. Now we prove the coefficient $32.449\log{q}$ similarly as we proved the previous estimate.  

The first term on the left hand side  in this new inequality  is a product of an increasing function $17.472x^{0.577}e/\left(\pi (x^{0.577}+7.509)\right)$, bounded by $17.472e/\pi$, and a decreasing function 
\begin{equation*}
\frac{\log(\log(q(T+1))}{\log(\log(qx))}.
\end{equation*}
Thus the coefficient $32.449\log{q}$ is proved.

Now we move on to the case $-1.720$. Since both 
\begin{equation*}
-\frac{12x^{0.577}e}{\pi (x^{0.577}+7.509)} \qquad \text{and} \qquad \frac{\log{\log{(q(T+1))}}}{\log{\log{(qx)}}}
\end{equation*} 
are decreasing functions for $x \geq 2$ and $q \geq 3$, we have
\begin{equation*}
-\frac{12x^{0.577}e\log{\log{(q(T+1))}}}{\pi (x^{0.577}+7.509)\log{\log{(qx)}}}\leq -\frac{12\cdot2^{0.577}e}{\pi (2^{0.577}+7.509)}<-1.720.
\end{equation*}
This completes the proof.
\end{proof}

We are almost ready with the estimates used in the proof of Theorem \ref{largeRho}. In the next lemma we combine the terms which go to zero when $x$ goes to infinity.
\begin{lemma}
\label{lemmaasympt0}
Assume $x \geq 2$, $q \geq 3$ and $T=x^{0.577}+8.509$. Then
\begin{align*}
& 12.624\frac{\sqrt{x}\log x}{T+1}+\frac{1.092\sqrt{x}}{T-1}\log{T}+\frac{4\sqrt{x}}{T-1}\log{\log{(qT)}}+\frac{\log x}{T+1}+\frac{1.092\sqrt{x}}{T-1}\log{q} \\
&\quad+0.893\frac{\sqrt{x}}{T+1}-1.250\frac{\sqrt{x}}{T-1}+\frac{1}{T+1}+\frac{xe}{\pi (T-1)\log{x}}\left(\frac{3\pi}{4(T-1)}+\frac{3}{(T-1)^2}\right) \\
& \quad+ \frac{3\log{\left(\frac{T+1}{2}\right)}+\log{(q\pi)}+3.570+\gamma+\frac{8}{T-1}}{\pi (T-1)x\log{x}} \\
&\quad <13.962\frac{\log{x}}{x^{0.077}}+\frac{8.4}{x^{0.077}}\log{\log{(qx)}}+\frac{1.255\log{q}+8.510}{x^{0.077}}.
\end{align*}
\end{lemma}
\begin{proof}
First we prove three useful inequalities used in this proof. First, by the proof of Lemma \ref{lemma0423} we have 
\begin{equation*}
\frac{\log{(T+1)}}{T-1}<0.577\frac{\log x}{x^{0.577}} \quad\text{and thus} \quad\frac{\log{(T+1)}}{T-1}\sqrt{x}<0.577\frac{\log x}{x^{0.077}}.
\end{equation*} 

Then, let us prove that $\log{\log{(qT)}}\leq 2.100\log{\log{(qx)}}$ for all $x \geq 2$ and $q \geq 3$. The derivative of the function $\log{\log{(qT)}}-2.100\log{\log{(qx)}}$ is
\begin{align*}
&\frac{0.577}{(x^{0.577} + 8.509) x^{0.423} \log(q (x^{0.577} + 8.509))} - \frac{2.1}{x \log(q x)} \\
&\quad<\frac{1}{x\log{(qx)}}- \frac{2.1}{x \log(q x)}<0.
\end{align*}
This means that the function $\log{\log{(qT)}}- 2.100\log{\log{(qx)}}$ is decreasing for $x$. Setting $x=2$ we get $\log{\log{(q\left(2^{0.577}+8.509\right))}}-2.100\log{\log{(2q)}}.$ The derivative of this function is 
\begin{equation*}
\frac{1}{q}\left(\frac{1}{\log{\left(q\left(2^{0.577}+8.509\right)\right)}}-\frac{2.1}{\log{(2q)}}\right),
\end{equation*}
which is negative for all $q \geq 3$, since we have $\left(q\left(2^{0.577}+8.509\right)\right)^{2.1}>2q$. Thus setting $q=3$ we see that the inequality $\log{\log{(qT)}}\leq 2.100\log{\log{(qx)}}$ holds for $x \geq 2$ and $q \geq 3$.

The last inequality needed is $\frac{\sqrt{x}}{T-1}>\frac{0.165}{x^{0.077}}$ for $x \geq 2$. It follows from the definition of the number $T$.

Using the previous three inequalities, the inequality under the consideration can be estimated with
\begin{align*}
&<12.624\frac{\log x}{x^{0.077}}+\frac{1.092\cdot0.577}{x^{0.077}}\log{x}+\frac{8.4}{x^{0.077}}\log{\log{(qx)}}+\frac{\log x}{x^{0.577}}+\frac{1.092}{x^{0.077}}\log{q} \\
&\quad+\frac{0.893}{x^{0.077}}-1.250\frac{0.165}{x^{0.077}}+\frac{1}{x^{0.577}}+\frac{e}{\pi \log{x}}\left(\frac{3\pi}{4x^{0.154}}+\frac{3}{x^{0.731}}\right)\\
& \quad+ \frac{3\cdot0.577\log{x}+\log{(q\pi)}-3\log{2}+3.570+\gamma+\frac{8}{x^{0.077}}}{\pi x^{1.577}\log{x}}. 
\end{align*} 
Setting $x=2$ on the unnecessary terms in the denominator, we get the result.
\end{proof}

At the end of this section, we prove one more inequality related to the function $\psi(x;q,a)$. It will be directly used in the proof of Theorem \ref{psi}.

\begin{lemma}
\label{lemmapsiInequalities}
Let $x \geq 2$ and $q \geq 3$. Then $\log{x}<4.778x^{0.077}$,
\begin{equation*}
x^{0.423}\log{\log{(qx)}}<4.778\cdot0.077\sqrt{x}+\left(0.077\log q+\log{4.778}\right)x^{0.423}
\end{equation*}
and
\begin{align*}
\left(\log\log{(qx)}\right)^2&<(0.077\cdot 9.556)^2x^{0.077}+(0.077\log q)^2+\log^2{4.778} \\
&\quad+2\left(0.077^2\cdot 9.556x^{0.0385}\log q \right. \\
&\left.\quad+0.077\cdot 9.556x^{0.0385}\cdot \log{4.778}+0.077\log{4.778}\log q\right).
\end{align*}
\end{lemma}
\begin{proof}
Taking a look at the derivative of the function $\log{x}-4.778x^{0.077}$, we see that the derivative is decreasing for $x \geq 2$. Furthermore, the derivative obtains its unique zero when 
\begin{equation*}
\frac{1}{x}=\frac{4.778\cdot 0.077x^{0.077}}{x},
\end{equation*} which holds when $x=1/(4.778\cdot 0.077)^{1/0.077}$. This is the point where the function $\log{x}-4.778x^{0.077}$ obtains its maximum, which is $<-0.0009$, and hence negative. This proves the first claim. 

Since the inequality $\log{x}<4.778x^{0.077}$ holds at a point $x=2$, it also holds for all numbers $x \geq 2$. Thus also the second inequality holds since we have
\begin{align*}
x^{0.423}\log{\log{(qx)}}&<x^{0.423}\log{(4.778(qx)^{0.077})} \\
&<4.778\cdot0.077\sqrt{x}+\left(0.077\log q+\log{4.778}\right)x^{0.423}.
\end{align*}

Let us move on to the third inequality. The function $\log{x}-9.556x^{0.077/2}$ obtains its maximum value at $x=\left(\frac{1}{0.367906}\right)^{1/0.0385}$ and thus it is always negative. 
Hence and using the first inequality, we have
\begin{align*}
\left(\log\log{(qx)}\right)^2&<\left(\log{(4.778(qx)^{0.077})}\right)^2  \\
&<\left(9.556\cdot0.077x^{0.0385}+\left(0.077\log q+\log{4.778}\right)\right)^2\\
&=(0.077\cdot 9.556)^2x^{0.077}+(0.077\log q)^2+\log^2{4.778} \\
&\quad+2\left(0.077^2\cdot 9.556x^{0.0385}\log q \right. \\
&\left.\quad+0.077\cdot 9.556x^{0.0385}\cdot \log{4.778}+0.077\log{4.778}\log q\right).
\end{align*}
\end{proof}

\subsection{Estimates related to the function $\pi(x;q,a)$}
\label{secInequPi}
In the proof of Theorem \ref{pi}, we get quite long estimates for the term $\psi(x;q,a)$ and hence we would like to simplify them. The next two inequalities are related to this. 

\begin{lemma}
\label{estForNegative}
Let $q \geq 3$ and $x \geq q$ be integers. Then we have 
\begin{equation*}
\frac{x^{1/4}}{\log x}>0.416\log\log x \quad \text{and}\quad x^{1/4}\log\log x<0.524\sqrt{x}.
\end{equation*} 
Further, if $x \geq q \geq  10^{29}$, we even have 
\begin{equation*}
\frac{x^{1/4}}{\log^2x}>949.261\log\log x.
\end{equation*}
\end{lemma}
\begin{proof}
Let us first consider the inequality $x^{1/4}/\log x> 0.416\log\log x$. We set $y=\log x$ and consider the function 
\begin{equation*}
f(y)=\frac{e^{y/4}}{y}- 0.416\log y.
\end{equation*} 
The derivative of the function is 
\begin{equation*}
f'(y)=\frac{e^{y/4} (-1 + 0.25 y) - 0.416 y}{y^2}.
\end{equation*}
Further, the denominator of the derivative is always positive and thus the sign of the derivative depends only on the numerator. 

The derivative of the numerator is $0.0625ye^{y/4}-0.416$. The previous derivative is positive for $y \geq 4 W(208/125) \approx 3.081$ and negative for $0<y < 4 W(208/125)$, where $W(x)$ is a Lambert $W$-function which corresponds to principal value. Thus the function $y^2f'(y)$ is increasing for $y \geq 4 W(208/125)$ and otherwise decreasing. Noticing that $6.192^2f'(6.192)>0$, we can conclude that the function $f(y)$ is increasing for $y \geq 6.192$.

Let us take a look at the values $y < 6.192$. Since $y \geq \log 3$, the function $y^2f'(y)$ is decreasing for 
\begin{equation*}
0<y < 4 W(208/125) \quad \text{and} \quad \log^2 3 f'(\log 3)<0,
\end{equation*} 
we see that the function $f(y)$ is decreasing for $\log3 \leq  y < 4 W(208/125)$. Furthermore, since the function $y^2f'(y)$ is increasing for $y>4 W(208/125)$ and we have 
\begin{equation*}
6.191^2f'(6.191)<0,
\end{equation*} 
the function $f(y)$ is decreasing (at least) for $y \leq 6.191$.

We have obtained that the function $f(y)$ is increasing for (at least) $y \geq 6.192$ and decreasing for (at least) $y \leq 6.191$. Thus the minimum value of the function $f(y)$ is obtained somewhere in $y \in [6.191, 6.192]$. In this interval, we have
\begin{equation*}
f(y)=\frac{e^{y/4}}{y}- 0.416\log y>\frac{e^{6.191/4}}{6.192}- 0.416\log 6.192\approx 0.001>0.
\end{equation*}
This proves the the inequality $x^{1/4}/\log x> 0.416\log\log x$.

Let us now move to the inequality $x^{1/4}\log\log x<0.524\sqrt{x}$. This one is simpler than the preceding one. Looking at the function 
\begin{equation*}
h(x)=\frac{0.524x^{1/4}}{\log\log x},
\end{equation*} 
we notice that 
\begin{equation*}
h'(x)=0.524\frac{\log(x)\log\log(x)-4}{4x^{3/4}\log x\log^2(\log x)},
\end{equation*} 
where the denominator is certainly positive. The numerator obtains its zero at $e^{e^{W(4)}}\approx e^{3.327322}$ which is in the interval $]27.863,27.864[$. Since 
\begin{equation*}
h\left(e^{e^{W(4)}}\right)>\frac{0.524\cdot 27.863^{1/4}}{\log\log 27.864}>1.001,
\end{equation*} 
the claim is proved.

Let us now move on to the inequality $x^{1/4}/\log^2x>949.261\log\log x$ for $x \geq q \geq  10^{29}$. Similarly as in the previous case, let us set
\begin{equation*}
y=\log x \geq 29\log{10} \approx 66.775
\end{equation*} 
and consider the function 
\begin{equation*}
g(y)=\frac{e^{y/4}}{949.261}-y^2\log y.
\end{equation*} 

The inequality follows quite similarly as in the previous case: We recognise that 
\begin{equation*}
g'''(y)=\frac{e^{y/4}}{4^3 \cdot 949.261}  - \frac{2}{y}
\end{equation*}
is positive for $y >4 W(30376.352)\approx 32.862$. Thus the function
\begin{equation*}
g''(y)=-3 + \frac{e^{y/4}}{4^2\cdot 949.261} - 2 \log y
\end{equation*}
is increasing for all $y$ under the consideration. Furthermore 
\begin{equation*}
g''(29\log{10}) \approx 1159.429>0
\end{equation*}
and thus the function 
\begin{equation*}
g'(y)=\frac{e^{y/4}}{4\cdot 949.261} - y - 2 y\log y
\end{equation*}
is increasing for all $y \geq 29 \log{10}$. Even more, $g'(29\log{10}) \approx 4055.464>0$ and hence the function $g(y)$ is increasing for all wanted $y$. Noticing that $g(29\log{10}) \approx 0.016>0$ we have proved the last inequality.
\end{proof}

Next we prove a lower bound for the terms with negative signs.
\begin{lemma}
\label{lemmaPiLowerSub}
Let $q \geq 3$ and $x \geq q$ be integers. Then formula
\begin{align*}
&\left(11.364\log q+284.488\right)\frac{x^{1/4}}{\log x}+\left(10.096\log^2 q+261.658\log{q}+2456.585\right)\frac{x^{1/4}}{\log^2x} \\ 
& \quad+\left(80.768\log^2 q+1753.168\log{q}+11605.056\right)\frac{x^{1/4}}{\log^3x} 
\end{align*}
can be estimated by
\begin{equation}
\label{kanselloiva}
> \left(2.015\log{q}+0.5\right)\log\log x +2.104\log^2q+55.018\log q+611.027
\end{equation}
for all $q$ and by
\begin{multline}
\label{kanselloiva2}
\left(\frac{0.297}{\log 2}\log^2q+2.015\log{q}+0.5\right)\log\log x \\
+45086.567\log^2q+2.8\cdot10^6\log q+8.5\cdot10^7
\end{multline}
for $q \geq 10^{29}$.
\end{lemma}

\begin{proof}
Let us first start with case \eqref{kanselloiva}. For $x\geq e$, the function $x^{1/4}/\log x$ obtains its minimum at $x=e^4$ and the minimum is $e/4$. Even more, by Lemma \ref{estForNegative} we have 
\begin{equation*}
\frac{x^{1/4}}{\log x}\geq 0.416\log\log x.
\end{equation*} 
Furthermore, the function $x^{1/4}/\log^2x$ obtains its minimum at $x=e^8$ and this minimum is $e^2/64$ and the function $x^{1/4}/\log^3x$ obtains its minimum at $x=e^{12}$, and this minimum is $e^3/1728$. Hence
\begin{align*}
&\left(11.364\log q+284.488\right)\frac{x^{1/4}}{\log x}+\left(10.096\log^2 q+261.658\log{q}+2456.585\right)\frac{x^{1/4}}{\log^2x} \\ 
& \quad+\left(80.768\log^2 q+1753.168\log{q}+11605.056\right)\frac{x^{1/4}}{\log^3x} \\
&\geq \left(11.364\log q-\frac{2.015}{0.416}\log q+284.488-\frac{0.5}{0.416}\right)\frac{e}{4}\\ 
&\quad +\left(10.096\log^2 q+261.658\log{q}+2456.585\right)\cdot \frac{e^2}{64}\\ 
& \quad+\left(80.768\log^2 q+1753.168\log{q}+11605.056\right) \cdot \frac{e^3}{1728} \\
&\quad+\left(2.015\log{q}+0.5\right)\log\log x\\ 
&> \left(2.015\log{q}+0.5\right)\log\log x +2.104\log^2q+55.018\log q+611.027.
\end{align*}

Let us now move on to case \eqref{kanselloiva2}. By Lemma \ref{estForNegative} we have 
\begin{equation*}
\frac{x^{1/4}}{\log^2x}>949.261\log\log x
\end{equation*} 
for $x \geq q\geq 10^{29}$. Furthermore, the functions 
\begin{equation*}
    \frac{x^{1/4}}{\log x},\quad \frac{x^{1/4}}{\log^2x} \quad \text{and} \quad \frac{x^{1/4}}{\log^3x}
\end{equation*} 
are increasing when $x \geq q \geq  10^{29}$. Thus, in this case, we write
\begin{align*}
&\left(11.364\log q+284.488\right)\frac{x^{1/4}}{\log x}+\left(10.096\log^2 q+261.658\log{q}+2456.585\right)\frac{x^{1/4}}{\log^2x} \\ 
& \quad+\left(80.768\log^2 q+1753.168\log{q}+11605.056\right)\frac{x^{1/4}}{\log^3x}  \\
&\quad\geq \left(11.364\log q-\frac{2.015}{0.416}\log q+284.488-\frac{0.5}{0.416}\right)\cdot \frac{10^{29/4}}{\log{(10^{29})}} \\ 
&\quad\quad +\left(\left(10.096-\frac{0.297}{949.261\log 2}\right)\log^2 q+261.658\log q+2456.585\right)\cdot \frac{10^{29/4}}{\log^2{(10^{29})}} \\
&\quad\quad +\left(80.768\log^2 q+1753.168\log{q}+11605.056\right)\cdot \frac{10^{29/4}}{\log^3{(10^{29})}} \\
&\quad\quad+\left(\frac{0.297}{\log 2}\log^2q+2.015\log{q}+0.5\right)\log\log x\\ 
&\quad> \left(\frac{0.297}{\log 2}\log^2q+2.015\log{q}+0.5\right)\log\log x \\
&\quad\quad+45086.567\log^2q+2.8\cdot10^6\log q+8.5\cdot10^7.
\end{align*}
\end{proof}

\section{Proof of Theorem \ref{psi}}
\label{secProofPsi}

The proof has borrowed its structure from Chapters 19 and 20 in Davenport's book \cite{davenport}. However, the details require different ideas, efficient use of newer results and ideas, and plenty of technical details.

Write now
\begin{equation*}
\psi(x;q,a)=\frac{1}{\varphi(q)}\sum_{\chi}\overline{\chi}(a)\psi(x,\chi),
\end{equation*}
where $\chi$ is a multiplicative character to the modulus $q$ and
\begin{equation*}
\psi(x,\chi)=\sum_{n\leq x}\chi(n)\Lambda(n).
\end{equation*}
Now we separate the sum corresponding to the principal character $\chi_0$ from the rest of the sum.

Considering the contribution coming from the principal character is straightforward.

\subsection{Contribution coming from the principal character}
We start with a lemma:
\begin{lemma}
\label{lemmaLambdaSyt}
Let $q\geq 3$. Now
\begin{equation*}
\sum_{\substack{n\leq x \\ (n,q)>1}}\Lambda(n)\leq \begin{cases}2\log x & \textrm{when }q=6 \\ \log q\log x & \textrm{when }q\ne 6.\end{cases}
\end{equation*}
\end{lemma}

\begin{proof}
To estimate this sum, notice first that if $q=\prod_{i=1}^k p_i^{\alpha_i}$, where $p_1<p_2<\dots <p_k$. Now
\begin{equation*}
\log q= \log \left(\prod_{i=1}^k p_i^{\alpha_i}\right)=\sum_{i=1}^k \alpha_i\log p_i.
\end{equation*}
Furthermore, $\alpha_i\log p_i>1$ whenever $p_i\geq 3$ or $\alpha_i\geq 2$ and in particular, if $p_i\geq 3$, then $\alpha_i\log p_i\geq \alpha_i$. Furthermore, $\log 2+\log p>2$ whenever $p\geq 5$ is a prime. Hence, $\log q\geq k$ whenever $q\ne 6$.

Assume now $q\ne 6$. Then
\begin{equation*}
\sum_{\substack{p\in\mathbb{P} \\ p|q}}\sum_{\substack{v \\ p^v\leq x}}\log{p}\leq \sum_{\substack{p\in\mathbb{P} \\ p|q}} \log p\left\lfloor\log_p x\right\rfloor\leq \sum_{\substack{p\in\mathbb{P} \\ p|q}} \log p\log_p x= \sum_{\substack{p\in\mathbb{P} \\ p|q}} \log x \leq \log x\log q.
\end{equation*}
If $q=6$, the sum is equal to
\begin{equation*}
\sum_{\substack{p\in\mathbb{P} \\ p|6}}\sum_{\substack{v \\ p^v\leq x}}\log{p}\leq \log 2\left\lfloor \log_2 x\right\rfloor+\log 3\left\lfloor \log_3 x\right\rfloor\leq 2\log x.
\end{equation*}
\end{proof}

Notice that 
\begin{equation*}
\psi(x,\chi_0)-\psi(x)=-\sum\limits_{\substack{n \leq x \\ (n,q)>1}} \Lambda(n),
\end{equation*}
where $\psi(x)=\sum_{n\leq x}\Lambda(n)$. Hence, by Lemma \ref{lemmaLambdaSyt},
\begin{equation}
\label{psiFormula}
\psi(x;q,a)=\frac{1}{\varphi(q)}\sum_{\chi}\overline{\chi}(a)\psi(x,\chi)=\frac{\psi(x)}{\varphi(q)}+\frac{c_1\log x}{\varphi(q)}+\frac{1}{\varphi(q)}\sum_{\chi\ne \chi_0}\overline{\chi}(a)\psi(x,\chi),
\end{equation}
where $-\log q\leq c_1 \leq 0$ if $q\ne 6$ and $-2\leq c_1 \leq 0$ if $q=6$. Hence, the contribution coming from the principal character is now treated. We may now move to considering the contribution coming from the other characters. This will require several technical lemmas.

\subsection{Contribution coming from the other characters}
We modify the function $\psi(x,\chi)$ to obtain the function $\psi_0(x,\chi)$ in the following way:
\begin{equation}
\label{formulaPsi0}
\psi_0(x,\chi)=\frac{1}{2}\left[\psi(x^+,\chi)+\psi(x^-,\chi)\right]
\end{equation}
because the function $\psi(x,\chi)$ has discontinuities when $x$ is a prime power, so we define the value to be the mean between the values on the left and right sides. This gives an error of size at most $\log x$.

Now we get:
\begin{lemma}
\label{psiNonPrimitive}
Let $\chi$ be an imprimitive Dirichlet character modulo $q\geq 3$ which is induced by a primitive character $\chi^*$. Then
\begin{equation*}
\left|\psi_0(x,\chi)-\psi_0(x,\chi^*)\right| \leq \begin{cases}2\log x & \textrm{ when }q=6\\ \log q\log x & \textrm{ when }q\ne 6.\end{cases}
\end{equation*}
\end{lemma}
\begin{proof} The estimate follows from the observation
\begin{equation*}
\left|\psi_0(x,\chi)-\psi_0(x,\chi^*)\right|\leq \sum_{\substack{n\leq x \\ (n,q)>1}}\Lambda(n)
\end{equation*}
and Lemma \ref{lemmaLambdaSyt}.
\end{proof}

According to the previous lemma, it suffices to estimate the contribution coming from primitive characters. Hence, suppose that $\chi$ is a primitive character. If $\chi(-1)=-1$, then (\cite{davenport}, Chapter 19, formula (2))
\begin{equation}
\label{psi0M1}
\psi_0(x,\chi)=-\sum_{\rho}\frac{x^{\rho}}{\rho}-\frac{L'(0,\chi)}{L(0,\chi)}+\sum_{m=1}^{\infty}\frac{x^{1-2m}}{2m-1}.
\end{equation}
If $\chi(-1)=1$, then  (\cite{davenport}, Chapter 19, formula (3))
\begin{equation}
\label{psi01}
\psi_0(x,\chi)=-\sum_{\rho}\frac{x^{\rho}}{\rho}-\log x-b(\chi)+\sum_{m=1}^{\infty}\frac{x^{-2m}}{2m},
\end{equation}
where $b(\chi)$ comes from the Laurent series of $\frac{L'(s,\chi)}{L(s,\chi)}$:
\begin{equation}
\label{bChiDef}
\frac{L'(s,\chi)}{L(s,\chi)}=\frac{1}{s}+b(\chi)+\cdots .
\end{equation}
We estimate the function $\psi_0(x,\chi)$ using the previous formulas. First we state two results which are useful to estimate the contribution. The first one is related to the number of nontrivial zeros of Dirichlet $L$-functions. Let $N(T,\chi)$ denote the zeros of the function $L(s,\chi)$ with $0<\Re\rho<1$ and $|\Im\rho|\leq T$. By \cite[Theorem 1.1]{bennett}:
\begin{theorem}
\label{NZeros}
Assume that $\chi$ be a character with conductor $q \geq 2$ and let $T \geq 5/7$. If $\log{\frac{q(T+2)}{2\pi}} \leq 1.567$, then $N(T,\chi)=0$. Otherwise we have
\begin{multline}
\label{NEqu}
\left|N(T,\chi)-\frac{T}{\pi}\log{\frac{qT}{2\pi e}}+\frac{\chi(-1)}{4}\right| \\
\leq 0.22737\log{\frac{q(T+2)}{2\pi}}+2\log{\left(1+\log{\frac{q(T+2)}{2\pi}}\right)}-0.5.
\end{multline}
\end{theorem}

The previous formula is quite long and sometimes a little bit difficult to use. Thus we sometimes use the following result proved in \cite[Corollary 1.2]{bennett}:
\begin{theorem}
\label{NZeros2}
Assume that $\chi$ be a character with conductor $q \geq 2$ and let $T \geq 5/7$. Then
\begin{equation*}
\left|N(T,\chi)-\frac{T}{\pi}\log{\frac{qT}{2\pi e}}\right|\leq 0.247\log{\frac{qT}{2\pi}}+6.894.
\end{equation*}
\end{theorem}

Next we estimate the logarithmic derivative of the function $L(s,\chi)$. Let $\gamma$ denote the Euler-Mascheroni constant.

\begin{lemma}
\label{L2tEst}
Let $s=1+\frac{1}{\log{y}}+it$, where $y>1$ and $t$ are real numbers. Then 
\begin{equation*}
\left|\frac{L'(s,\chi)}{L(s,\chi)}\right|<\log{y}+\gamma+\frac{0.478}{\log{y}}.
\end{equation*}
Furthermore, if $s=\sigma+it$, where $\sigma \geq 2$, then $\left|\frac{L'(s,\chi)}{L(s,\chi)}\right|<0.570$.
\end{lemma}
\begin{proof}
We can compute
\begin{equation*}
\left|\frac{L'(s,\chi)}{L(s,\chi)}\right|=\left|\sum_{n=1}^\infty \frac{\Lambda(n)\chi(n)}{n^{1+\frac{1}{\log{y}}}e^{it\log{n}}}\right|\le\sum_{n=1}^\infty\frac{\Lambda(n)}{n^{1+\frac{1}{\log{y}}}}=\left|\frac{\zeta'(1+\frac{1}{\log{y}})}{\zeta(1+\frac{1}{\log{y}})}\right|. 
\end{equation*}
By \cite[Lemma 2.2]{broughan} the right hand side on the previous formula is
\begin{equation*}
<\log{y}+\gamma+\frac{0.478}{\log{y}}.
\end{equation*}
Furthermore, since $\left|\frac{\zeta'(\sigma)}{\zeta(\sigma)}\right|<0.570$, where $\sigma \geq 2$, we have $\left|\frac{L'(s,\chi)}{L(s,\chi)}\right|<0.570$ for $s=\sigma+it$.
\end{proof}

Next we estimate the function $\psi_0(x,\chi)$ using formulas \eqref{psi0M1} and \eqref{psi01}. First we estimate the term $b(\chi)$.

\begin{lemma}
\label{Lemmabchi}
Assume that the term $b(\chi)$ is defined as in formula \eqref{bChiDef}, $\chi$ is a primitive non-principal character modulo $q$ with $\chi(-1)=1$ and $L(s,\chi)$ satisfies GRH. Then
\begin{equation*}
\left|b(\chi)\right|<2.751\log{q}+23.878.
\end{equation*}
\end{lemma}

\begin{proof}
To estimate the term $b(\chi)$ we would like to find a formula for it. Since the term $b(\chi)$ comes from the Laurent series of $\frac{L'(s,\chi)}{L(s,\chi)}$, we would like to write it in the form where we can find the term $b(\chi)$ easily. By the functional equation for Dirichlet $L$-functions and logarithmic differentation we have
\begin{equation}
\label{eqLlogarithmic}
\frac{L'(s,\chi)}{L(s,\chi)}=-\frac{1}{2}\log{\frac{q}{\pi}}-\frac{1}{2}\frac{\Gamma'\left(\frac{s}{2}\right)}{\Gamma\left(\frac{s}{2}\right)}+B(\chi)+\sum_\rho\left(\frac{1}{s-\rho}+\frac{1}{\rho}\right),
\end{equation}
where $B(\chi)$ is a constant which depends on the character $\chi$ and the sum is over the nontrivial zeros of the function $L(s, \chi)$. We do not want to evaluate the term $B(\chi)$ and thus we we want to remove it. When we subtract formula \eqref{eqLlogarithmic} with $s=2$ from formula with $s$, we obtain
\begin{equation*}
\frac{L'(s,\chi)}{L(s,\chi)}=\frac{L'(2,\chi)}{L(2,\chi)}-\frac{1}{2}\frac{\Gamma'\left(\frac{s}{2}\right)}{\Gamma\left(\frac{s}{2}\right)}+\frac{1}{2}\frac{\Gamma'\left(1\right)}{\Gamma\left(1\right)}+\sum_\rho\left(\frac{1}{s-\rho}-\frac{1}{2-\rho}\right).
\end{equation*}
By \cite[Section 12, formula (9)]{davenport}
\begin{equation}
\label{eqGamma}
-\frac{\Gamma'\left(z\right)}{\Gamma\left(z\right)}=\gamma+\frac{1}{z}+\sum\limits_{n=1}^\infty\left(\frac{1}{z+n}-\frac{1}{n}\right).
\end{equation}
Since $b(\chi)$ is the value of the function $\frac{L'(s,\chi)}{L(s,\chi)}-\frac{1}{s}$ at $s=0$, we have
\begin{equation}
\label{bChi}
b(\chi)=\frac{L'(2,\chi)}{L(2,\chi)}-\sum_\rho\left(\frac{1}{\rho}+\frac{1}{2-\rho}\right).
\end{equation}
To estimate the term $b(\chi)$, we can estimate the two terms on the right hand side of the previous equation. The first term is estimated in Lemma \ref{L2tEst} and thus we only need to estimate the second term.

Now we estimate the term $\sum_\rho\left(\frac{1}{\rho}+\frac{1}{2-\rho}\right)$. This term can be written as
\begin{equation*}
\sum_\rho\left(\frac{1}{\rho}+\frac{1}{2-\rho}\right)=\sum_{\substack{\rho \\ |\Im\rho|\le5/7}}\frac{2}{\rho(2-\rho)}+\sum_{\substack{\rho \\ |\Im\rho|>5/7}}\frac{2}{\rho(2-\rho)}.
\end{equation*}
First we estimate the first term on the right hand side of the previous equation and then we estimate the second term. Since we assume the GRH for the function $L(s,\chi)$, we have $\left|\rho(2-\rho)\right|\geq \frac{3}{4}$. Thus
\begin{equation*}
\left|\sum_{\substack{\rho \\ |\Im\rho|\leq 5/7}}\frac{2}{\rho(2-\rho)}\right|\leq \frac{8}{3}N(5/7,\chi).
\end{equation*}
We can estimate the sum $\sum\limits_{\substack{\rho \\ |\Im\rho|>5/7}}\frac{2}{\rho(2-\rho)}$ similarly. By Theorem \ref{NZeros2}, we have
\begin{align*}
\left|\sum_{\substack{\rho \\ |\Im\rho|>5/7}}\frac{2}{\rho(2-\rho)}\right| &\leq 4\int_{5/7}^\infty\frac{1}{t^3}\left(N(t,\chi)-N(5/7,\chi)\right) dt \\
& \leq -3.920N(5/7,\chi) \\
&\quad+\int\limits_{5/7}^\infty\frac{4}{t^3}\left(\frac{t}{\pi}\log{\frac{qt}{2\pi e}}+0.247\log{qt} +6.894\right)dt\\ 
&\leq -3.920N(5/7,\chi)+4\left[-\frac{\log(t)+1}{t\pi}-\frac{1}{t\pi}\log\left(\frac{q}{2\pi e}\right)\right.\\ 
& \quad \left.-0.247\left(\frac{2\log(t)+1}{4t^2}+\frac{\log q}{2t^2}\right)-\frac{6.894}{2t^2}\right]_{5/7}^{\infty}\\ 
& \leq -3.920N(5/7,\chi)+2.751\log{q}+23.308.
\end{align*}

Thus we have estimated
\begin{equation}
\label{sumzeros2Est}
\left|\sum_\rho\left(\frac{1}{\rho}+\frac{1}{2-\rho}\right) \right| <2.751\log{q}+23.308.
\end{equation}
The claim follows from formulas \eqref{bChi} and \eqref{sumzeros2Est} and Lemma \ref{L2tEst}.
\end{proof}

Next we estimate the contribution coming from the terms $\sum_{m=1}^{\infty}\frac{x^{1-2m}}{2m-1}$ and $\sum_{m=1}^{\infty}\frac{x^{-2m}}{2m}$. It is sufficient to consider these sums in the case $x\geq 2$ because if $x<2$, then $\psi(x,\chi)=0$ and $\psi(x;q,a)=0$. Because of the same reason, many of the later results are proved for $x \geq 2$.

\begin{lemma} \label{sumx} Assume $x\geq 2$. Then
\begin{equation*}
\sum_{m=1}^{\infty}\frac{x^{1-2m}}{2m-1}=\frac{1}{2}\log\left(1+\frac{2}{x-1}\right)\leq 1
\end{equation*}
and
\begin{equation*}
\sum_{m=1}^{\infty}\frac{x^{-2m}}{2m}=-\frac{1}{2}\log\left(1-\frac{1}{x^2}\right)\leq \frac{1}{6}.
\end{equation*}
\end{lemma}
\begin{proof} We have
\begin{equation*}
\sum_{m=1}^{\infty}\frac{x^{1-2m}}{2m-1}=\left[\int \sum_{m=1}^{\infty}t^{2m-2}dt\right]_{t=1/x}=\left[\int \frac{dt}{1-t^2}\right]_{t=1/x}=\frac{1}{2}\log\left(1+\frac{2}{x-1}\right)
\end{equation*}
and
\begin{equation*}
\frac{1}{2}\log\left(1+\frac{2}{x-1}\right)=\frac{1}{2}\int_1^{1+2/(x-1)}\frac{dt}{t}\leq \frac{1}{2}\cdot \frac{2}{x-1}\cdot 1=\frac{1}{x-1}\leq1.
\end{equation*}
We may now move to the other sum:
\begin{align*}
\sum_{m=1}^{\infty}\frac{x^{-2m}}{2m}&=\left[\sum_{m=1}^{\infty}\frac{t^{2m}}{2m}\right]_{t=1/x}=\int_0^{1/x}\sum_{m=1}^{\infty}t^{2m-1}dt \\ &=\int_0^{1/x}\frac{t}{1-t^2}dt=-\frac{1}{2}\log\left(1-\frac{1}{x^2}\right).
\end{align*}
Furthermore,
\begin{equation*}
-\frac{1}{2}\log\left(1-\frac{1}{x^2}\right)=\frac{1}{2}\int_{1-1/x^2}^1\frac{dt}{t}\leq \frac{1}{2}\cdot \frac{1}{x^2}\cdot \frac{1}{1-1/x^2}\leq \frac{1}{2(x^2-1)}\leq \frac{1}{6}.
\end{equation*}
\end{proof}

Since in formula \eqref{psi0M1} we have the term $\frac{L'(0,\chi)}{L(0,\chi)}$ with $\chi(-1)=-1$, we have to estimate it.

\begin{lemma}
\label{logarithmicL}
Assume $\chi$ is a primitive character modulo $q$, $q\ge3$, $\chi(-1)=-1$ and GRH holds for $L(s,\chi)$. Then
\begin{equation*}
\left|\frac{L'(0,\chi)}{L(0,\chi)}\right| < 0.020q^{0.5216}\log{q}+0.005q^{0.5216}+239.330q^{0.0216}+\left|\log{\frac{q}{\pi}}\right|+\gamma+\log{2},
\end{equation*}
if $3\leq q<4\cdot 10^5$,
\begin{equation*}
\left|\frac{L'(0,\chi)}{L(0,\chi)}\right| <2.259q^{0.0216}\log^2{q}+\log{\frac{q}{\pi}}+\gamma+\log{2},
\end{equation*}
if $4\cdot 10^5\leq q<10^{10}$ and
\begin{align*}
\left|\frac{L'(0,\chi)}{L(0,\chi)}\right| & < \frac{3.28272e^\gamma}{\pi^2}\left(\log\log{q}-\log{2}+\frac{1}{2}+\frac{1}{\log\log{q}}+\frac{14\log\log{q}}{\log{q}}\right)\log^2{q} \\
&\quad+\log{\frac{q}{\pi}}+\gamma+\log{2},
\end{align*}
if $q\geq 10^{10}$.
\end{lemma}
\begin{proof}
We have
\begin{equation*}
\left|\frac{L'(0,\chi)}{L(0,\chi)}\right| =\left|\log{\frac{\pi}{q}}-\frac{\Gamma'(1)}{2\Gamma(1)}-\frac{\Gamma'(0.5)}{2\Gamma(0.5)}-\frac{L'(1,\bar{\chi})}{L(1,\bar{\chi})}\right|\leq \left|\log{\frac{q}{\pi}}\right|+\gamma+\log{2}+\left|\frac{L'(1,\bar{\chi})}{L(1,\bar{\chi})}\right|.
\end{equation*}
Thus it is sufficient to estimate the term $\left|\frac{L'(1,\bar{\chi})}{L(1,\bar{\chi})}\right|$. We divide the proof to different cases depending on the size of $q$. First we assume that $q \geq 10^{10}$. Then by \cite[Lemma 6.5]{theta} and \cite[Theorem 1.5]{lamzouri} the right hand side on the previous equality is
\begin{multline*}
< \frac{3.28272e^\gamma}{\pi^2}\left(\log\log{q}-\log{2}+\frac{1}{2}+\frac{1}{\log\log{q}}+\frac{14\log\log{q}}{\log{q}}\right)\log^2{q} \\
+\log{\frac{q}{\pi}}+\gamma+\log{2}.
\end{multline*}
Next we assume that $3\leq q <10^{10}$. First we estimate the term $\frac{1}{\left|L(1,\bar{\chi})\right|}$. In this proof, we are going to make two choices for parameters. The first one is done now, and the second one slightly later. These choices do not need to be really optimised because they affect the least important term in the main results, and overall contribution coming from these terms will be small. Therefore, these choices are not really optimized, but we have picked constants which yield sufficiently good bounds. By \cite[Lemma 2.5, Lemma 2.3]{lamzouri} for $x=150$ we have
\begin{align*}
\log{\left|L(1,\bar{\chi})\right|} &\geq \Re\left(\sum_{n\leq 150}\frac{\bar{\chi}(n)\Lambda(n)\log{\frac{150}{n}}}{n\log{n}\log{150}}\right)+\frac{1}{2\log{150}}\left(\log{\frac{q}{\pi}}-\gamma\right) \\
& \quad-\left(\frac{1}{\log{150}}+\frac{2}{\sqrt{150}\log^2{150}}\right)\left(1-\frac{1}{\sqrt{150}}\right)^{-2}\left(\vphantom{\Re\left(\sum_{n\le150}\frac{\bar{\chi}(n)\Lambda(n)}{n}\left(1-\frac{n}{150}\right)\right)}\frac{149}{300}\log{\frac{q}{\pi}} \right.\\
& \quad \left. -\Re\left(\sum_{n\le150}\frac{\bar{\chi}(n)\Lambda(n)}{n}\left(1-\frac{n}{150}\right)\right)\right.\\
&\quad\left.\vphantom{\Re\left(\sum_{n\le150}\frac{\bar{\chi}(n)\Lambda(n)}{n}\left(1-\frac{n}{150}\right)\right)}-\sum_{k=0}^\infty\frac{150^{-2k-2}}{(2k+1)(2k+2)}-\frac{149\gamma}{300}+\frac{\log{2}}{150}\right)-\frac{2}{150\log^2 150}  \\
& \geq -\sum_{n\leq 150}\frac{\Lambda(n)\log{\frac{150}{n}}}{n\log{n}\log{150}}+\frac{1}{2\log{150}}\left(\log{\frac{q}{\pi}}-\gamma\right) \\
& \quad-\left(\frac{1}{\log{150}}+\frac{2}{\sqrt{150}\log^2{150}}\right)\left(1-\frac{1}{\sqrt{150}}\right)^{-2}\left(\frac{149}{300}\log{\frac{q}{\pi}} \right. \\
& \quad \left. +\sum_{n\le150}\frac{\Lambda(n)}{n}\left(1-\frac{n}{150}\right)-\sum_{k=0}^\infty\frac{150^{-2k-2}}{(2k+1)(2k+2)}-\frac{149\gamma}{300}+\frac{\log{2}}{150}\right)\\
&\quad-\frac{2}{150\log^2 150}.
\end{align*}
Using Sage \cite{sage}, we have
\begin{equation*}
-\sum_{n\leq 150}\frac{\Lambda(n)\log{\frac{150}{n}}}{n\log{n}\log{150}}\approx -1.30397> -1.304
\end{equation*}
and
\begin{equation*}
\sum_{n\le150}\frac{\Lambda(n)}{n}\left(1-\frac{n}{150}\right)\approx 3.44556< 3.446.
\end{equation*}
Furthermore
\begin{equation*}
-\sum_{k=0}^\infty\frac{150^{-2k-2}}{(2k+1)(2k+2)}< 0.
\end{equation*}

We have $\log{\left|L(1,\bar{\chi})\right|} > -0.0216\log{q}-2.111$. Thus we have proved 
\begin{equation}
\label{estL1}
\frac{1}{\left|L(1,\bar{\chi})\right|}< e^{2.111}q^{0.0216}.
\end{equation} 
Next we estimate the term $\left|L'(1,\bar{\chi})\right|$. By \cite[Lemma 6.5]{theta} for $4\cdot 10^5\leq q<10^{10}$ we have
\begin{equation}
\label{estL'1Large}
\left|L'(1,\bar{\chi})\right| < 0.27356\log^2{q}.
\end{equation}
Furthermore, by \cite[Lemma 6.4]{theta} with $y=2000$ for $3\leq q<4\cdot10^5$, we have
\begin{equation}
\label{estL'1Small}
\left|L'(1,\bar{\chi})\right| \leq \frac{\sqrt{q}\log{2000}}{1000\pi}\log{\frac{4q}{\pi}}+\frac{\log^2{2000}}{2}+\frac{1}{10}.
\end{equation}
Thus by estimates \eqref{estL1}, \eqref{estL'1Large} and \eqref{estL'1Small} we have
\begin{equation*}
\left|\frac{L'(1,\bar{\chi})}{L(1,\bar{\chi})}\right|  <2.259q^{0.0216}\log^2{q}  
\end{equation*}
if $4\cdot 10^5\leq q<10^{10}$ and
\begin{equation*}
\left|\frac{L'(1,\bar{\chi})}{L(1,\bar{\chi})}\right|<
0.020q^{0.5216}\log{q}+0.005q^{0.5216}+239.330q^{0.0216} 
\end{equation*}
if $3\leq q<4\cdot 10^5$.
Thus we have proved the claim.
\end{proof}

Observe that using the bound $q^{0.0216}\leq (4\cdot 10^5)^{0.0216}\approx1.321309\dots $ on the interval $3\leq q<4\cdot 10^5$, the bound $q^{0.0216}\approx 1.64437\dots $ on the interval $4\cdot 10^5\leq q<10^{10} $, and the bound
\begin{equation*}
\frac{3.28272e^\gamma}{\pi^2}\left(\log\log{q}-\log{2}+\frac{1}{2}+\frac{1}{\log\log{q}}+\frac{14\log\log{q}}{\log{q}}\right)<0.593\log\log q+1.205,
\end{equation*}
for $q\geq 10^{10}$, we obtain the following somewhat simpler but slightly weaker bounds:

\begin{corollary}
\label{Lcorollary}
Assume $\chi$ is a primitive character modulo $q$, $q\ge3$, $\chi(-1)=-1$ and GRH holds for $L(s,\chi)$. Then
\begin{equation*}
\left|\frac{L'(0,\chi)}{L(0,\chi)}\right| < 0.027\sqrt{q}\log{q}+0.067\sqrt{q}+316.229+\left|\log{\frac{q}{\pi}}\right|+\gamma+\log{2},
\end{equation*}
if $3\leq q<4\cdot 10^5$,
\begin{equation*}
\left|\frac{L'(0,\chi)}{L(0,\chi)}\right| <3.715\log^2{q}+\log{\frac{q}{\pi}}+\gamma+\log{2},
\end{equation*}
if $4\cdot 10^5\leq q<10^{10}$ and
\begin{equation*}
\left|\frac{L'(0,\chi)}{L(0,\chi)}\right|< (0.593\log\log q+1.205)\log^2{q}+\log{\frac{q}{\pi}}+\gamma+\log{2},
\end{equation*}
if $q\geq 10^{10}$.
\end{corollary}

Furthermore, in our results we use induced characters and they may have a smaller modulus than the original character. Hence we prove the following lemma:
\begin{lemma}
\label{Lremark}
Let $q^*$ and $q$ be integers such that $3 \leq q^* \leq q$. Let $\chi^*$ be a primitive non-principal character of modulo $q^*$. Then
\begin{equation*}
\left|\frac{L'(0,\chi^*)}{L(0,\chi^*)}\right| < 0.027\sqrt{q}\log{q}+0.067\sqrt{q}+316.229+\log{(q \pi)}+\gamma+\log{2}
\end{equation*}
if $3\leq q<4\cdot 10^5$,
\begin{equation*}
\left|\frac{L'(0,\chi^*)}{L(0,\chi^*)}\right| <3.715\log^2{q}+\log{\frac{q}{\pi}}+\gamma+\log{2},
\end{equation*}
if $4\cdot 10^5\leq q<10^{29}$ and
\begin{equation*}
\left|\frac{L'(0,\chi)}{L(0,\chi)}\right|< (0.593\log\log q+1.205)\log^2{q}+\log{\frac{q}{\pi}}+\gamma+\log{2},
\end{equation*}
if $q\geq 10^{29}$.
\end{lemma}
\begin{proof}
Let us start with the case $3\leq q<4\cdot 10^5$. The function
\begin{equation*}
0.027\sqrt{y}\log{y}+0.067\sqrt{y}+316.229+\log{(y \pi)}+\gamma+\log{2}
\end{equation*}
is an increasing function. Hence, if $q < 4\cdot 10^5$, by Corollary \ref{Lcorollary} we can estimate
\begin{equation*}
\left|\frac{L'(0,\chi^*)}{L(0,\chi^*)}\right| < 0.027\sqrt{q}\log{q}+0.067\sqrt{q}+316.229+\log{(q \pi)}+\gamma+\log{2}.
\end{equation*}

Let us now move on to the case $4\cdot 10^5\leq q<10^{29}$. First notice that the second upper bound is greater than the first one in Corollary \ref{Lcorollary} when $4\cdot 10^5\leq q <676924$ and the function
\begin{equation}
\label{estLSecond}
3.715\log^2{y}+\log{\frac{y}{\pi}}+\gamma+\log{2}
\end{equation}
is increasing. Furthermore, by Corollary \ref{Lcorollary} for all $q^* < 4\cdot10^5$ we have 
\begin{equation*}
    \left|L'(0,\chi^*)/L(0,\chi^*)\right|<594.189.
\end{equation*}
On the other hand, for all $y \geq 676924$ formula \eqref{estLSecond} is greater than $683.139$. Thus the case $4\cdot 10^5\leq q^*<10^{10}$ is proved. We still have to consider the case $10^{10}\leq q^* <10^{29}$ before we can finish the case $4\cdot 10^5\leq q<10^{29}$. When we compare formula
\begin{equation}
\label{estLThird}
(0.593\log\log y+1.205)\log^2{y}+\log{\frac{y}{\pi}}+\gamma+\log{2}
\end{equation}
to formula \eqref{estLSecond}, we see that formula \eqref{estLSecond} gives a larger estimate for $10^{10}\leq y <10^{29}$. Hence we can use the estimate \eqref{estLSecond} with $y=q$ for $10^{10}\leq q^* <10^{29}$ too.

Now we have only the case $q \geq 10^{29}$ left. We notice that the function \eqref{estLThird} is increasing and greater than $16548$ for $y \geq 10^{29}$. Hence and by the previous paragraph, for $3 \leq q^* < 4\cdot10^5$ we have
\begin{equation*}
\left|\frac{L'(0,\chi^*)}{L(0,\chi^*)}\right|  < (0.593\log\log q+1.205)\log^2{q}+\log{\frac{q}{\pi}}+\gamma+\log{2}.
\end{equation*}
Further, we notice that since
\begin{equation*}
    3.715\log^210^{10}+\log\frac{10^{10}}{\pi}<(0.593\log\log q+1.205)\log^2{q}+\log{\frac{q}{\pi}},
\end{equation*}
the claim is proved.
\end{proof}

Thus the only term from formulas \eqref{psi0M1} and \eqref{psi01} which is not estimated yet is the term $\sum_{\rho}\frac{x^{\rho}}{\rho}$. To obtain the estimate we first prove some preliminary results.

\subsubsection{Preliminaries for contribution coming from the term $\sum_{\rho}\frac{x^{\rho}}{\rho}$}

In this section we prove results which are used to estimate the function $\sum_{\rho}\frac{x^{\rho}}{\rho}$. The main goal is to estimate the terms $\frac{L'(s,\chi)}{L(s,\chi)}$ at the cases where $L(s,\chi) \ne 0$. We start by proving a lemma about the spacing of the zeros, namely, that it is possible to find a horizontal line which is sufficiently far away from the zeros of the function $L(s,\chi)$. 
\begin{lemma}
\label{lemmaDifference}
Assume that $\chi$ is a primitive non-principal character modulo $q \geq 3$ and $T \geq 10$. Then there are numbers $T_1\in (T-1,T+1]$ and $T_2 \in [-T-1,-T+1)$  such that for $T_j$,  $j \in \{1,2\}$, it holds that
\begin{equation*}
\left|\Im\rho-T_j\right|>\frac{1}{1.092\log{(qT)}+4\log{\log{(qT)}}-0.250}
\end{equation*}
for all nontrivial zeros of the function $L(s,\chi)$.
\end{lemma}
\begin{proof}
By Theorem \ref{NZeros} and the bounds 
\begin{equation*}
\frac{T}{\pi}\log{\left(1+\frac{2}{T-1}\right)}\leq \frac{10}{\pi} \log\frac{11}{9}
\end{equation*} 
and 
\begin{equation*}
\log{(T+1)(T+3)\leq \log{\frac{143T^2}{100}}},
\end{equation*} there are at most
\begin{equation}
\label{estimateNZeros}
\begin{aligned}
& \frac{T+1}{\pi}\log{\frac{q(T+1)}{2\pi e}}-\frac{\chi(-1)}{4}+0.22737\log{\frac{q(T+3)}{2\pi}} \\
& \quad +2\log{\left(1+\log{\frac{q(T+3)}{2\pi}}\right)}-0.5-\frac{T-1}{\pi}\log{\frac{q(T-1)}{2\pi e}}+\frac{\chi(-1)}{4} \\
&\quad+0.22737\log{\frac{q(T+1)}{2\pi}}+2\log{\left(1+\log{\frac{q(T+1)}{2\pi}}\right)}-0.5 \\
& \quad = \frac{T}{\pi}\log{\left(1+\frac{2}{T-1}\right)}+\frac{1}{\pi}\log{\left(q^{2}(T^2-1)\right)} \\
&\quad\quad+0.22737\log{\left(q^2(T+1)(T+3)\right)}+2\left(\log{\left(\log{\frac{qe(T+3)}{2\pi}}\right)}\right. \\
&\quad\quad\left.+\log{\left(\log{\frac{qe(T+1)}{2\pi}}\right)}\right)+\left(-\frac{2}{\pi}+0.45474\right)\log{(2\pi)}-\frac{2}{\pi}-1 \\
& \quad <2\left(0.22737+\frac{1}{\pi}\right)\log{(qT)}+4\log{\left(\log{\frac{13qeT}{20\pi}}\right)}\\ 
& \quad\quad+\frac{10}{\pi} \log\frac{11}{9}+0.22737\log{\frac{143}{100}}+\left(-\frac{2}{\pi}+0.45474\right)\log{(2\pi)}-\frac{2}{\pi}-1 \\
& \quad <1.092\log{(qT)}+4\log{\log{(qT)}}-1.250.
\end{aligned}
\end{equation}
zeros $\rho$ with $0<\Re\rho<1$ and $|\Im\rho|\in (T-1,T+1]$. Thus there are numbers $T_1, -T_2 \in (T-1,T+1]$ such that for $T_0 \in \{T_1,T_2\}$ it holds that
\begin{equation*}
\left|\Im\rho-T_0\right|>\frac{1}{1.092\log{(qT)}+4\log{\log{(qT)}}-1.250+1}
\end{equation*}
for all nontrivial zeros of the function $L(s,\chi)$ and the claim follows.
\end{proof}

Now we can apply the previous result and estimate one useful sum which is used later to estimate the function $\frac{L'(s,\chi)}{L(s,\chi)}$.

\begin{lemma}
\label{sumSmallDifference}
Assume that $\chi$ is a primitive non-principal character modulo $q \geq 3$ and $T \geq 10$. Further assume that $L(s,\chi)$ satisfies the GRH. Furthermore, let $\Im s$ be either $T_1$ or $T_2$ in Lemma \ref{lemmaDifference}. Then
\begin{multline*}
\sum_{\substack{\rho \\ |\Im\rho- \Im s|< 1}}\frac{3}{|\sigma+ i\Im s-\rho||2+ i\Im s-\rho|}< 2.385\log^2{\left(q(T+1)\right)} \\
+17.472\log{\left(q(T+1)\right)}\log{\log{(q(T+1))}}-3.276\log{\left(q(T+1)\right)}\\
 +32\left(\log{\log{(q(T+1))}}\right)^2-12\log{\log{(q(T+1))}}+0.625,
\end{multline*}
where the sum runs over the nontrivial zeros of the function $L(s,\chi)$ with $|\Im\rho- \Im s|< 1$.
\end{lemma}
\begin{proof}
By Lemma \ref{lemmaDifference}, we have
\begin{equation*}
\left|\sigma+ i\Im s-\rho\right|>\frac{1}{1.092\log{(qT)}+4\log{\log{(qT)}}-0.250},
\end{equation*}
and there are at most $1.092\log{(qT)}+4\log{\log{(qT)}}-1.250$ nontrivial zeros in $|\Im\rho-\Im s|<1$. Thus, when $T \geq 10$ and we assume GRH, we have
\begin{align*}
& \sum_{\substack{\rho \\ |\Im\rho- \Im s|< 1}}\frac{3}{|\sigma+ i\Im s-\rho||2+ i\Im s-\rho|} \\
& \quad<2\left(1.092\log{(q(T+1))}+4\log{\log{(q(T+1))}}-1.250\right) \cdot\\
&\quad\quad \cdot\left(1.092\log{(qT)}+4\log{\log{(qT)}}-0.250\right) \\
& \quad <  2.385\log^2{\left(q(T+1)\right)}+17.472\log{\left(q(T+1)\right)}\log{\log{(q(T+1))}} \\
&\quad\quad-3.276\log{\left(q(T+1)\right)}+32\left(\log{\log{(q(T+1))}}\right)^2\\
&\quad\quad-12\log{\log{(q(T+1))}}+0.625. 
\end{align*}
\end{proof}

Next we estimate a similar sum as the one above but where the sum runs over large values of the difference $|\Im\rho-\Im s|$. This result is also used to estimate the term $\frac{L'(s,\chi)}{L(s,\chi)}$. Let $\mathfrak{a}=1$ if $\chi(-1)=-1$ and $\mathfrak{a}=0$ if $\chi(-1)=1$.

\begin{lemma}
\label{sumLargeDifference}
Assume that $\chi$ is a primitive non-principal character modulo $q$ and $T \geq 10$. Further assume that $L(s,\chi)$ satisfies the GRH. Furthermore, let $\Im s$ be either $T_1$ or $T_2$ in Lemma \ref{lemmaDifference}. Then
\begin{equation*}
\sum_{\substack{\rho \\ |\Im\rho-\Im s|\geq 1}}\frac{3}{|\Im\rho-\Im s|^2}<\frac{13}{8}\log{\left(\frac{T^2}{4}+\frac{T}{2}+\frac{5}{2}\right)}+10.868\log q+78.232,
\end{equation*}
where the sum runs over the nontrivial zeros of the function $L(s,\chi)$ with $|\Im\rho-\Im s|\geq 1$.
\end{lemma}
\begin{proof} We first need the bound
\begin{equation*}
\sum_{\substack{\rho \\ |\Im\rho-\Im s|\geq 1}}\frac{3}{|\Im\rho-\Im s|^2}\leq \sum_{\substack{\rho \\ |\Im\rho-\Im s|\geq 1}}\frac{13}{2}\Re\left(\frac{1}{2+i\Im s-\rho}\right)
\end{equation*}
which follows from directly comparing the terms in the sums:
\begin{equation*}
\frac{13}{2}\Re\left(\frac{1}{2+i\Im s-\rho}\right)=\frac{13\cdot 3}{4}\cdot \frac{1}{\left(\frac{3}{2}\right)^2+(\Im\rho-\Im s)^2}\geq \frac{3}{|\Im\rho-\Im s|^2}
\end{equation*}
whenever $|\Im\rho-\Im s|\geq 1$. And since $\Re\frac{1}{\rho}=\frac{1}{2|\rho|^2}>0$, we have
\begin{align*}
\sum_{\substack{\rho \\ |\Im\rho-\Im s|\geq 1}}\frac{3}{|\Im\rho-\Im s|^2}&\leq \sum_{\substack{\rho \\ |\Im\rho-\Im s|\geq 1}}\frac{13}{2}\Re\left(\frac{1}{2+i\Im s-\rho}\right) \\
& < \sum_{\rho}\frac{13}{2}\Re\left(\frac{1}{2+i\Im s-\rho}+\frac{1}{\rho}\right).
\end{align*}
Further, by the functional equation for the Dirichlet $L$-functions, we get
\begin{align*}
\sum_{\rho}\Re\left(\frac{1}{2+i\Im s-\rho}+\frac{1}{\rho}\right)&=\Re\left(\frac{L'(2+ i\Im s,\chi)}{L(2+ i\Im s, \chi)}\right)+\frac{1}{2}\log{\frac{q}{\pi}} \\
&\quad+\Re\left(\frac{1}{2}\frac{\Gamma'\left(1+\frac{i\Im s}{2}+\frac{1}{2}\mathfrak{a}\right)}{\Gamma\left(1+\frac{i\Im s}{2}+\frac{1}{2}\mathfrak{a}\right)}\right)-\Re\left(B(\chi)\right),
\end{align*}
where $\Re\left(B(\chi)\right)=-\sum_\rho \Re\frac{1}{\rho}$. The first term on the right hand side of the previous formula can be estimated by Lemma \ref{L2tEst} and the second term is a constant. Furthermore, by \cite[Lemma 2.3]{chandee} the third term can be estimated to be
\begin{align*}
\Re\left(\frac{1}{2}\frac{\Gamma'\left(1+\frac{i\Im s}{2}+\frac{1}{2}\mathfrak{a}\right)}{\Gamma\left(1+\frac{i\Im s}{2}+\frac{1}{2}\mathfrak{a}\right)}\right)& \leq \frac{1}{2}\log \left|1+\frac{i\Im s}{2}+\frac{1}{2}\mathfrak{a}\right| \\ &\leq \frac{1}{2}\log \sqrt{\left(\frac{3}{2}\right)^2+\left(\frac{T+1}{2}\right)^2}\\ &\leq \frac{1}{4}\log \left(\frac{T^2}{4}+\frac{T}{2}+\frac{5}{2}\right).
\end{align*}
Thus we only need to estimate the term $\Re\left(B(\chi)\right)$.

Since we assume the GRH, we have
\begin{equation*}
\Re\left(B(\chi)\right)=-\sum_\rho \Re\frac{1}{\rho} = -\sum_{\substack{\rho \\ |\Im\rho|\leq 5/7}} \frac{1}{2|\rho|^2}-\sum_{\substack{\rho \\ |\Im\rho|> 5/7}} \frac{1}{2|\rho|^2}.
\end{equation*}
Furthermore, $\sum\limits_{\substack{\rho \\ |\Im\rho|\leq 5/7}} \frac{1}{2|\rho|^2}\leq 2N(5/7,\chi)$ and by Theorem \ref{NZeros2}, we have
\begin{align*}
\sum_{\substack{\rho \\ |\Im\rho|> 5/7}} \frac{1}{2|\rho|^2} &\leq  -\frac{49N(5/7,\chi)}{50}+ \int\limits_{5/7}^\infty\frac{1}{t^3}\left(\frac{t}{\pi}\log{\frac{qt}{2\pi e}}+0.247\log{qt} +6.894\right)dt\\ 
&\leq -\frac{N(9,\chi)}{162}+\left[-\frac{\log(t)+1}{t\pi}-\frac{1}{t\pi}\log\left(\frac{q}{2\pi e}\right)\right.\\ & \quad \left.-0.247\left(\frac{2\log(t)+1}{4t^2}+\frac{\log q}{2t^2}\right)-\frac{6.894}{2t^2}\right]_{5/7}^{\infty}\\ 
& <-\frac{49N(5/7,\chi)}{50}+0.688\log{q}+5.827.
\end{align*} 
Applying Theorem \ref{NZeros2} we have proved
\begin{align*}
-\Re(B(\chi)) &<\frac{51}{50}N(5/7,\chi)+0.688\log{q}+5.827 \\
&\leq \frac{51}{50}\left(\frac{5}{7\pi}\log{\frac{5q}{14\pi e}}+0.247\log{(5q/7)} +6.894\right)+0.688\log{q}+5.827 \\
& <1.172\log{q}+12.038.
\end{align*}

Combining the previous computations and applying Lemma \ref{L2tEst} we get
\begin{align*}
\sum_{\substack{\rho \\ |\Im\rho-\Im s|\geq 1}}\frac{3}{|\Im\rho-\Im s|^2} &<\frac{13}{2}\left(\frac{1}{4}\log{\left(\frac{T^2}{4}+\frac{T}{2}+\frac{5}{2}\right)}+\frac{1}{2}\log{q}+1.172\log{q}\right. \\
& \quad\quad\left. +12.038+0.570-\frac{1}{2}\log{\pi}\right).
\end{align*}
The claim follows from the previous formula.
\end{proof}

Now we apply the previous results to estimate the term $\frac{L'(s, \chi)}{L(s,\chi)}$ for some numbers $s$. First we estimate it when $\Re s$ lies in certain interval and after that for $\Re s$ small enough.

\begin{lemma}
\label{est1c}
Assume that the $L(s,\chi)$ satisfies the GRH and that $\chi$, $T$ and $\Im s$ satisfy the same assumptions as in Lemma \ref{sumSmallDifference} and $\Re s \in [-1,c]$, where $c=1+\frac{1}{\log{x}}$, $x\geq 2$. Then
\begin{multline*}
\left|\frac{L'(s, \chi)}{L(s,\chi)}\right| < 2.385\log^2{\left(q(T+1)\right)}+17.472\log{\left(q(T+1)\right)}\log{\log{(q(T+1))}}\\
-3.276\log{\left(T+1\right)} +\frac{13}{8}\log{\left(\frac{T^2}{4}+\frac{T}{2}+\frac{5}{2}\right)} +32\left(\log{\log{(q(T+1))}}\right)^2\\
-12\log{\log{(q(T+1))}}+7.592\log q+79.427+\frac{3\pi}{4(T-1)}+\frac{3}{(T-1)^2}.
\end{multline*}
\end{lemma}
\begin{proof}
 By the functional equation for Dirichlet $L$-functions with a primitive non-principal character
\begin{equation}
\label{Lformula}
\begin{aligned}
\frac{L'(s, \chi)}{L(s,\chi)} &=\frac{L'(2+ i\Im s, \chi)}{L(2+i\Im s, \chi)}-\frac{1}{2}\frac{\Gamma'\left(\frac{1}{2}s+\frac{1}{2}\mathfrak{a}\right)}{\Gamma\left(\frac{1}{2}s+\frac{1}{2}\mathfrak{a}\right)}+\frac{1}{2}\frac{\Gamma'\left(1+\frac{i\Im s}{2}+\frac{1}{2}\mathfrak{a}\right)}{\Gamma\left(1+\frac{i\Im s}{2}+\frac{1}{2}\mathfrak{a}\right)} \\
& \quad+\sum_\rho\left(\frac{1}{s-\rho}-\frac{1}{2+ i\Im s-\rho}\right).
\end{aligned}
\end{equation}
Since the first term on the right hand side of the previous formula can be estimated by Lemma \ref{L2tEst}, it is sufficient to estimate the last three terms of the previous formula. 

First we estimate the second and the third term of formula \eqref{Lformula}. Using formula \eqref{eqGamma}
we can compute
\begin{align*}
& -\frac{1}{2}\frac{\Gamma'\left(\frac{1}{2}s+\frac{1}{2}\mathfrak{a}\right)}{\Gamma\left(\frac{1}{2}s+\frac{1}{2}\mathfrak{a}\right)}+\frac{1}{2}\frac{\Gamma'\left(1+\frac{i\Im s}{2}+\frac{1}{2}\mathfrak{a}\right)}{\Gamma\left(1+\frac{i\Im s}{2}+\frac{1}{2}\mathfrak{a}\right)} \\
& \quad=\frac{1}{s+\mathfrak{a}}+\sum_{n=1}^\infty\left(\frac{1}{s+\mathfrak{a}+2n}-\frac{1}{2n}\right)\\
&\quad\quad-\frac{1}{2+i\Im s+\mathfrak{a}}-\sum_{n=1}^\infty\left(\frac{1}{2+ i\Im s+\mathfrak{a}+2n}-\frac{1}{2n}\right).
\end{align*}
Furthermore, the right hand side of the previous formula is
\begin{equation*}
=\frac{2-\Re s}{(s+\mathfrak{a})(2+i\Im s+\mathfrak{a})}+\sum\limits_{n=1}^\infty\frac{2-\Re s}{(s+\mathfrak{a}+2n)(2+i\Im s+\mathfrak{a}+2n)}.
\end{equation*}
By the series expansion of the function $\tanh(z)$ (\cite{GR}, 1.421 formula 2) and since $\Re s \in [-1,c]$ and $|\Im s|\in(T-1,T+1]$, the absolute value of the previous formula is
\begin{align*}
& < 3\left(\frac{1}{(T-1)^2}+\sum\limits_{n=1}^\infty \frac{1}{(T-1)^2+(2n-1)^2}\right) \\
& \quad =3\left(\frac{1}{(T-1)^2}+\frac{\pi}{4(T-1)}\tanh{\left(\frac{\pi}{2}(T-1)\right)}\right) \\
& \quad <3\left(\frac{1}{(T-1)^2}+\frac{\pi}{4(T-1)}\right).
\end{align*}
Thus we have estimated the second and the third term on the right hand side of formula \eqref{Lformula}. It is sufficient to estimate the last term of formula \eqref{Lformula}. First we notice that
\begin{align*}
&\left| \sum_\rho\left(\frac{1}{s-\rho}-\frac{1}{2+ i\Im s-\rho}\right)\right| \\
&\quad\le\sum_{\substack{\rho \\ |\Im\rho-\Im s|< 1}}\frac{3}{|s-\rho||2+ i\Im s-\rho|}+\sum_{\substack{\rho \\ |\Im\rho-\Im s|\geq 1}}\frac{3}{|\Im\rho-\Im s|^2}.
\end{align*}
By the assumptions for the number $\Im s$, Lemma \ref{sumSmallDifference} and \ref{sumLargeDifference}, the right hand side of the previous formula is 
\begin{align*}
& < 2.385\log^2{\left(q(T+1)\right)}+17.472\log{\left(q(T+1)\right)}\log{\log{(q(T+1))}}\\
&\quad-3.276\log{\left(q(T+1)\right)} +32\left(\log{\log{(q(T+1))}}\right)^2-12\log{\log{(q(T+1))}}+0.625 \\
&\quad+\frac{13}{8}\log{\left(\frac{T^2}{4}+\frac{T}{2}+\frac{5}{2}\right)}+10.868\log q+78.232.
\end{align*}

By \eqref{Lformula}, Lemma \ref{L2tEst} and combining the previous computations we get 
\begin{align*}
\left|\frac{L'(s, \chi)}{L(s,\chi)}\right| &< 2.385\log^2{\left(q(T+1)\right)}+17.472\log{\left(q(T+1)\right)}\log{\log{(q(T+1))}}\\
&\quad-3.276\log{\left(q(T+1)\right)} +32\left(\log{\log{(q(T+1))}}\right)^2\\
&\quad-12\log{\log{(q(T+1))}}+0.625+\frac{13}{8}\log{\left(\frac{T^2}{4}+\frac{T}{2}+\frac{5}{2}\right)}\\ 
& \quad +10.868\log q+78.232+0.570+\frac{3\pi}{4(T-1)}+\frac{3}{(T-1)^2}.\\
&=  2.385\log^2{\left(q(T+1)\right)}+17.472\log{\left(q(T+1)\right)}\log{\log{(q(T+1))}}\\
&\quad-3.276\log{\left(T+1\right)} +32\left(\log{\log{(q(T+1))}}\right)^2-12\log{\log{(q(T+1))}}\\ 
& \quad +\frac{13}{8}\log{\left(\frac{T^2}{4}+\frac{T}{2}+\frac{5}{2}\right)}+7.592\log q\\
&\quad+79.427+\frac{3\pi}{4(T-1)}+\frac{3}{(T-1)^2}.
\end{align*}
\end{proof}

Next we estimate the term $\frac{L'(s, \chi)}{L(s,\chi)}$ in the case when the absolute value of the imaginary part of the number $s$ is large enough and the real part is small enough.

\begin{lemma}
\label{LN1}
Assume that $\chi$ is a primitive non-principal character modulo $q$, $\Re s<-1$ and $|\Im s|=T>2$. Then 
\begin{equation*}
\left|\frac{L'(s,\chi)}{L(s,\chi)} \right|
<\left|\log{\frac{q}{\pi}}\right|+ \frac{3}{2}\log{\left(\frac{|1-s+\mathfrak{a}|}{2}\right)}+\frac{3}{2}\log{\left(\frac{|s+\mathfrak{a}|}{2}\right)}+3.570+\gamma+\frac{8}{T}.
\end{equation*}
\end{lemma}
\begin{proof}
By the functional equation for the Dirichlet $L$-functions, we have
\begin{equation}
\label{funcED}
\frac{L'(s,\chi)}{L(s,\chi)}=-\log{\frac{q}{\pi}}-\frac{1}{2}\frac{\Gamma'\left(\frac{1-s+\mathfrak{a}}{2}\right)}{\Gamma\left(\frac{1-s+\mathfrak{a}}{2}\right)}-\frac{1}{2}\frac{\Gamma'\left(\frac{s+\mathfrak{a}}{2}\right)}{\Gamma\left(\frac{s+\mathfrak{a}}{2}\right)}-\frac{L'(1-s,\bar{\chi})}{L(1-s,\bar{\chi})}.
\end{equation}
The first term does not depend on $s$ and the last term can be estimated by Lemma \ref{L2tEst}. Thus it is sufficient to estimate the second and the third term. Next we derive explicit estimates containing logarithmic derivatives of the gamma function for these terms.

The first term in formula \eqref{eqGamma}, which describes the logarithmic derivative of the $Gamma$ function, is a positive constant and the second term can be estimated by $|z^{-1}|\leq T^{-1}$ if $|\Im(z)|=T$. Thus we only estimate the infinite sum. Since on formula \eqref{funcED} the imaginary parts of the numbers $z$ satisfy $|\Im(z)|=T$, we derive the estimate for $|\Im(z)|=T>2$.

We divide the last term of the previous formula to three sums depending on the size of the index $n$. Let 
\begin{equation*}
N_1=\left\lfloor\frac{|z|}{2}\right\rfloor-1\ge0 \quad\text{and} \quad N_2=\left\lceil\frac{|z|}{2}\right\rceil+1.
\end{equation*}
By the definitions of the numbers $N_1$ and $N_2$, we have
\begin{equation}
\label{GammaSum}
\begin{aligned}
\left|\sum_{n=1}^\infty\frac{z}{2n(z+2n)}\right| &\leq \sum_{n=1}^{N_1} \frac{|z|}{2n(|z|-2n)} \\
&\quad+\sum_{n \in \left\{\left\lfloor\frac{|z|}{2}\right\rfloor, \left\lceil\frac{|z|}{2}\right\rceil \right\}} \left|\frac{z}{2n(z+2n)} \right|+\sum_{n=N_2}^\infty \frac{|z|}{2n(2n-|z|)}.
\end{aligned}
\end{equation}
Let us start with the term in the middle. Since $T>2$, we always have $|z|>2$. Assume first $|z|<4$. Then $\left\lfloor \frac{|z|}{2}\right\rfloor =1$ and $\left\lceil \frac{|z|}{2}\right\rceil =2$ and
\begin{equation*}
\sum_{n \in \left\{\left\lfloor\frac{|z|}{2}\right\rfloor, \left\lceil\frac{|z|}{2}\right\rceil \right\}} \left|\frac{z}{2n(z+2n)} \right|=\left|\frac{z}{2(z+2)} \right|+\left|\frac{z}{4(z+4)} \right|\leq \frac{|z|}{T}\left(\frac{1}{2}+\frac{1}{4}\right)<\frac{4}{T}\cdot \frac{3}{4}=\frac{3}{T}.
\end{equation*}
Assume now $|z|\geq 4$. Then
\begin{equation*}
\sum_{n \in \left\{\left\lfloor\frac{|z|}{2}\right\rfloor, \left\lceil\frac{|z|}{2}\right\rceil \right\}} \left|\frac{z}{2n(z+2n)} \right| <\frac{|z|}{2\left(\frac{|z|}{2}-1\right)T}+\frac{|z|}{2\frac{|z|}{2}T}=\frac{1}{T}\left(\frac{|z|}{|z|-2}+1\right)\leq \frac{3}{T}.
\end{equation*}
Hence, the term in the middle is always at most $\frac{3}{T}$.

First term:
\begin{align*}
 \sum_{n=1}^{N_1} \frac{|z|}{2n(|z|-2n)}&=\sum_{n=1}^{N_1}\frac{1}{2n}+\sum_{n=1}^{N_1}\frac{1}{|z|-2n}\leq \sum_{n=1}^{N_1}\frac{1}{2n}+\sum_{n=1}^{N_1}\frac{1}{2N_1+2-2n}\\ &=2\sum_{n=1}^{N_1}\frac{1}{2n}\leq 1+\int_1^{N_1}\frac{dx}{x}=1+\log N_1.
\end{align*}

Let us now move to the third term. Denote $f(x)=\frac{|z|}{2x(2x-|z|)}$. Then
\begin{equation*}
f'(x)=\frac{|z|((|z|-4x)}{2x^2(2x-|z|)^2}<0
\end{equation*}
when $x> \frac{|z|}{4}$ which is true when $x\geq N_2$. Hence
\begin{align*}
\sum_{n=N_2}^\infty \frac{|z|}{2n(2n-|z|)}&<\frac{|z|}{2N_2(2N_2-|z|)}+\frac{|z|}{2}\int_{N_2}^{\infty}\frac{dx}{x(2x-|z|)}\\
&=\frac{|z|}{2N_2(2N_2-|z|)}+\frac{|z|}{2}\left[\frac{\log(2x-|z|)-\log(x)}{|z|}\right]_{N_2}^{\infty}\\ 
&=\frac{|z|}{2N_2(2N_2-|z|)}-\frac{1}{2}\log(2N_2-|z|)+\frac{1}{2}\log N_2+\frac{1}{2}\log(2)\\ 
&\leq \frac{1}{2}+\frac{1}{2}\log N_2.
\end{align*}
Since $|z|>2$, we may now estimate
\begin{equation*}
\log N_1+\frac{1}{2}\log N_2\leq\frac{1}{2}\log \left(\left(\frac{|z|}{2}-1\right)^2\left(\frac{|z|}{2}+2\right)\right)<\frac{3}{2}\log \frac{|z|}{2},
\end{equation*}
and we have now obtained
\begin{equation*}
\left|\sum_{n=1}^\infty\frac{z}{2n(z+2n)}\right|<\frac{3}{2}\log \frac{|z|}{2}+\frac{3}{2}+\frac{3}{T}
\end{equation*}
and we have estimated the right hand side of formula \eqref{GammaSum}. Putting everything together, we obtain the estimate.
\end{proof}

We apply the results which are proved in this section to estimate the term $\sum_{\rho}\frac{x^{\rho}}{\rho}$.

\subsubsection{Contribution coming from the term $\sum_{\rho}\frac{x^{\rho}}{\rho}$}

In this section we estimate the term $\sum_{\rho}\frac{x^{\rho}}{\rho}$. First we estimate the error term which comes when we estimate the term $\sum_{\Im\rho\leq T+1}\frac{x^{\rho}}{\rho}$ instead. After that we estimate the term $\sum_{\Im\rho\leq T+1}\frac{x^{\rho}}{\rho}$.

Next two Lemmas are prepare for estimating the contribution. By \cite[Chapter 17, Lemma]{davenport} we have:
\begin{lemma}
\label{Iest}
Let 
\begin{equation*}
\delta(y)=
\begin{cases}
0 & \text{if } 0<y<1 \\
\frac{1}{2} & \text{if } y=1 \\
1 & \text{if } y>1
\end{cases}
\end{equation*} 
and
\begin{equation*}
I(y,T)=\frac{1}{2\pi i}\int_{c-iT}^{c+iT}\frac{y^s}{s} ds.
\end{equation*}
Then for $y>0$, $c>0$, $T>0$
\begin{equation*}
\left| I(y,T)-\delta(y) \right|<
\begin{cases}
y^c\min\left\{1,T^{-1}\left|\log{y}\right|^{-1}\right\} & \text{if } y\ne 1 \\
cT^{-1} & \text{if } y=1.
\end{cases}
\end{equation*}
\end{lemma}

We remember that it is sufficient to consider the cases $x \geq 2$ since if $x<2$, then $\psi(x,\chi)=0$ and $\psi(x;q,a)=0$. Next we prove one result where we estimate the term $\psi_0(x,\chi)$ by an integral.

\begin{lemma}
\label{Jest}
Let $T>0$, $x \geq 2$ and $c=1+\frac{1}{\log{x}}$. Then
\begin{align*}
\left|\psi_0(x,\chi)-J(x,T, \chi) \right|  &< 1.363\frac{x\log^2 x}{T}+2.074\sqrt{x}\log x+12.294\frac{x\log x}{T} +7.032\frac{x}{T}\\ & \quad+5.823\frac{x}{T\log x}+12.624\frac{\sqrt{x}\log x}{T}+0.893\frac{\sqrt{x}}{T}+\frac{\log x}{T}+\frac{1}{T}
\end{align*}
where
\begin{equation*}
J(x,T,\chi)=\frac{1}{2\pi i}\int_{c-iT}^{c+iT} \left(-\frac{L'(s,\chi)}{L(s,\chi)}\right)\frac{x^s}{s} ds.
\end{equation*}
\end{lemma}

\begin{proof}
Keep in mind that
\begin{align*}
J(x,T,\chi)&=\frac{1}{2\pi i}\int_{c-iT}^{c+iT} \left(-\frac{L'(s,\chi)}{L(s,\chi)}\right)\frac{x^s}{s} ds\\
&=\frac{1}{2\pi i}\sum\limits_{n=1}^\infty \Lambda(n)\chi(n)\left(\int_{c-iT}^{c+iT}\frac{(x/n)^s}{s}\right) ds.
\end{align*}
Hence, by the definition of the function $\psi_0(x,\chi)$ and Lemma \ref{Iest} we have
\begin{equation}
\label{psiest}
\left|\psi_0(x,\chi)-J(x,T,\chi) \right| <\sum\limits_{\substack{n=1 \\ n \ne x}}^\infty \Lambda(n)\left(\frac{x}{n}\right)^c\min\left\{1,T^{-1}\left|\log{\frac{x}{n}}\right|^{-1}\right\}+cT^{-1}\Lambda(x). 
\end{equation}

We want to estimate this sum by first considering the terms far from $x$ and then considering the terms close to $x$. First we notice that $x^c=ex$ and when $n \leq \frac{4}{5}x$ or $n \geq \frac{5}{4}x$, we have $\left|\log{\frac{x}{n}}\right|\geq \log{\frac{5}{4}}$. Thus the right hand side of the formula \eqref{psiest} is

\begin{equation}
\label{psisum}
\begin{aligned}
& 
\leq \frac{xe}{T\log{\frac{5}{4}}}\left(\sum_{n\leq \frac{4}{5}x}\frac{\Lambda(n)}{n^c}+\sum_{n \geq \frac{5}{4}x}\frac{\Lambda(n)}{n^c}\right) \\
& \quad
+\sum_{\substack{\frac{4}{5}x<n<\frac{5}{4}x \\ n \ne x}} \Lambda(n)\left(\frac{x}{n}\right)^c\min\left\{1,T^{-1}\left|\log{\frac{x}{n}}\right|^{-1}\right\}+cT^{-1}\Lambda(x).
\end{aligned}
\end{equation}

First we estimate the first term on the right hand side of the previous formula. By \cite[Lemma 2.2]{broughan} and since $c=1+\frac{1}{\log{x}}$, we have 
\begin{multline}
\label{psifirst}
\frac{xe}{T\log{\frac{5}{4}}}\left(\sum_{n\leq \frac{4}{5}x}\frac{\Lambda(n)}{n^c}+\sum_{n \geq \frac{5}{4}x}\frac{\Lambda(n)}{n^c}\right) \\ \le\frac{xe}{T\log{\frac{5}{4}}} \left|\frac{\zeta'(c)}{\zeta(c)}\right|<\frac{xe}{T\log{\frac{5}{4}}}\left(\log{x}+\gamma+\frac{0.478}{\log{x}}\right).
\end{multline}
Thus we have estimated the first term of formula \eqref{psisum}. Let us now move to the second term of
formula \eqref{psisum}, namely, to the terms with  $\frac{4}{5}x<n<\frac{5}{4}x$.

We start with terms $x-\frac{\sqrt{x}}{5}<n<x+\frac{\sqrt{x}}{4}$. 
We have
\begin{align*}
& \sum_{\substack{x-\frac{\sqrt{x}}{5}<n<x+\frac{\sqrt{x}}{4} \\ n \ne x}} \Lambda(n)\left(\frac{x}{n}\right)^c\min\left\{1,T^{-1}\left|\log{\frac{x}{n}}\right|^{-1}\right\} \\
&\quad\leq \sum_{\substack{x-\frac{\sqrt{x}}{5}<n<x+\frac{\sqrt{x}}{4} \\ n \ne x}} \Lambda(n)\left(\frac{x}{n}\right)^c\\ 
&\quad\leq \log\left(x+\frac{\sqrt{x}}{4}\right)x^c\sum_{\substack{x-\frac{\sqrt{x}}{5}<n<x+\frac{\sqrt{x}}{4} \\ n \ne x}}n^{-c}\\ 
&\quad\leq \log\left(x+\frac{\sqrt{x}}{4}\right)x^c\left(\frac{\sqrt{x}}{4}+\frac{\sqrt{x}}{5}+1\right)\left(x-\frac{\sqrt{x}}{5}\right)^{-c}\\ 
&\quad\leq \sqrt{x}\log\left(x+\frac{\sqrt{x}}{4}\right)\left(\frac{1}{4}+\frac{1}{5}+\frac{1}{\sqrt{x}}\right)\left(1-\frac{1}{5\sqrt{x}}\right)^{-c}.
\end{align*}

It is clear that when $x\geq 2$, the terms $\left(\frac{1}{4}+\frac{1}{5}+\frac{1}{\sqrt{x}}\right)$ and $\left(1-\frac{1}{5\sqrt{x}}\right)^{-c}$ obtain their largest values when $x=2$. Furthermore, the function $\frac{\log\left(x+\frac{\sqrt{x}}{4}\right)}{\log x}$ is decreasing when $x\geq 2$, and therefore, it also obtains its maximum when $x=2$. For the expression above, we get therefore the estimate
\begin{equation}
\label{psimiddle}
\leq 2.074\sqrt{x}\log x.
\end{equation}

We need to estimate sums including the  term  $\left|\log{\frac{x}{n}}\right|^{-1}$, so before moving to the actual estimates for sums, let us bound the logarithm.

When $x+\frac{\sqrt{x}}{4}\leq n\leq \frac{5}{4}x$, we have
\begin{equation*}
\left|\log\frac{x}{n}\right|=\log \frac{n}{x}= \log \left(1+\frac{n-x}{x}\right)=\int_{1}^{1+\frac{n-x}{x}}\frac{dy}{y}\geq	 \frac{n-x}{n}.
\end{equation*}

Bounding the logarithm and using the bound
\begin{equation*}
n\left(\frac{x}{n}\right)^c=\frac{xe}{n^{\frac{1}{\log{x}}}}<\frac{xe}{n^{\frac{1}{\log{n}}}}= x, \quad \text{for} \quad n>x,
\end{equation*}
we have
\begin{align*}
&\sum_{x+\frac{\sqrt{x}}{4}\leq n\leq \frac{5}{4}x} \Lambda(n)\left(\frac{x}{n}\right)^c\min\left\{1,T^{-1}\left|\log{\frac{x}{n}}\right|^{-1}\right\}\\
&\quad\leq T^{-1}\sum_{x+\frac{\sqrt{x}}{4}\leq n\leq \frac{5}{4}x} \Lambda(n)\left(\frac{x}{n}\right)^c\left|\log{\frac{x}{n}}\right|^{-1}\\ 
&\quad\leq \log\left(\frac{5x}{4}\right)T^{-1}x^c\sum_{x+\frac{\sqrt{x}}{4}\leq n\leq \frac{5}{4}x}n\left(\frac{x}{n}\right)^c\cdot \frac{n}{n-x} \\
&\quad\leq 
\log\left(\frac{5x}{4}\right)T^{-1}x\sum_{x+\frac{\sqrt{x}}{4}\leq n\leq \frac{5}{4}x}\frac{1}{n-x}.
\end{align*}

Bounding the sum with an integral yields 
\begin{equation}
\label{psilarge}
\begin{aligned}
& \leq \log\left(\frac{5x}{4}\right)T^{-1}x\left(\frac{4}{\sqrt{x}}+\int_{x+\frac{\sqrt{x}}{4}}^{\frac{5}{4}x} \frac{1}{n-x} dn\right) \\
& \quad=\log\left(\frac{5x}{4}\right)T^{-1}x\left(\frac{4}{\sqrt{x}}+\log\left(\frac{x}{4}\right)-\log\left(\frac{\sqrt{x}}{4}\right)\right) \\
& \quad=\log\left(\frac{5x}{4}\right)T^{-1}x\left(\frac{4}{\sqrt{x}}+\frac{\log x}{2}\right).
\end{aligned}
\end{equation}

Look at the terms $\frac{4}{5}x\leq n\leq x-\frac{\sqrt{x}}{5}$. Now
\begin{equation*}
\left|\log\frac{x}{n}\right|=\log \frac{x}{n}= \log \left(1+\frac{x-n}{n}\right)=\int_{1}^{1+\frac{x-n}{n}}\frac{dy}{y}\geq \frac{x-n}{x}.
\end{equation*}
Bounding the logarithm, we get
\begin{align*}
&\sum_{\frac{4}{5}x\leq n\leq x-\frac{\sqrt{x}}{5}} \Lambda(n)\left(\frac{x}{n}\right)^c\min\left\{1,T^{-1}\left|\log{\frac{x}{n}}\right|^{-1}\right\} \\
&\quad\leq T^{-1}\sum_{\frac{4}{5}x\leq n\leq x-\frac{\sqrt{x}}{5}} \Lambda(n)\left(\frac{x}{n}\right)^c\left|\log{\frac{x}{n}}\right|^{-1}\\
&\quad \leq xT^{-1}\left(\frac{5}{4}\right)^{c}\log x\sum_{\frac{4}{5}x\leq n\leq x-\frac{\sqrt{x}}{5}}\frac{1}{x-n}.
\end{align*}
 
Bounding the last sum with an integral yields 
\begin{equation}
\label{psismall}
\begin{aligned}
&\leq xT^{-1}\left(\frac{5}{4}\right)^{c}\log x\left(\frac{5}{\sqrt{x}}+\int_{\frac{4}{5}x}^{x-\frac{\sqrt{x}}{5}} \frac{1}{x-n} dn\right) \\
& \quad\leq xT^{-1}\left(\frac{5}{4}\right)^{c}\log(x)\left(\frac{5}{\sqrt{x}}+\log \frac{x}{5}-\log \frac{\sqrt{x}}{5}\right) \\
& \quad\leq xT^{-1}\left(\frac{5}{4}\right)^{c}\log(x)\left(\frac{5}{\sqrt{x}}+\frac{1}{2}\log x\right).
\end{aligned}  
\end{equation}

Combining the results from formulas \eqref{psiest}, \eqref{psisum}, \eqref{psifirst}, \eqref{psimiddle}, \eqref{psilarge} and \eqref{psismall}, we have
\begin{align*}
&\left|\psi_0(x,\chi)-J(x,T,\chi) \right|\\
&\quad<cT^{-1}\Lambda(x)+\frac{xe}{T\log{\frac{5}{4}}}\left(\log{x}+\gamma+\frac{0.478}{\log{x}}\right) x+2.074\sqrt{x}\log\\
&\quad\quad+\log\left(\frac{5x}{4}\right)T^{-1}x\left(\frac{4}{\sqrt{x}}+\frac{\log x}{2}\right)+xT^{-1}\left(\frac{5}{4}\right)^{c}\log(x)\left(\frac{5}{\sqrt{x}}+\frac{1}{2}\log x\right)\\ 
&\quad\leq 1.363\frac{x\log^2 x}{T}+2.074\sqrt{x}\log x+12.294\frac{x\log x}{T}+7.032\frac{x}{T}+5.823\frac{x}{T\log(x)}\\
&\quad\quad+12.624\frac{\sqrt{x}\log(x)}{T}+0.893\frac{\sqrt{x}}{T}+\frac{\log x}{T}+\frac{1}{T}.
\end{align*}
\end{proof}

Using previous results we can estimate the error term which comes when we estimate the term $\sum_{\rho}\frac{x^{\rho}}{\rho}$ with the term $\sum_{\Im\rho\leq T+1}\frac{x^{\rho}}{\rho}$.

Now we are finally ready to move on and cut the sum $\sum_{\rho}\frac{x^\rho}{\rho}$ to the sum $\sum_{|\Im\rho| \leq T+1}\frac{x^\rho}{\rho}$:
\begin{theorem}
\label{largeRho}
Assume that $\chi$ is a primitive non-principal character modulo $q$, $L(s,\chi)$ satisfies GRH, $x\geq 2$ and $T=x^{0.577}+8.509$. Then
\begin{align*}
& \left|\psi_0(x,\chi)+\sum_{\substack{\rho \\ |\Im\rho|\leq T+1}} \frac{x^\rho}{\rho}+\mathfrak{a}\frac{L'(0,\chi)}{L(0,\chi)}+(1-\mathfrak{a})\left(\log{x}+b(\chi)\right)-\sum_{m=1}^\infty\frac{x^{\mathfrak{a}-2m}}{2m-\mathfrak{a}}\right| \\
&\quad<1.363x^{0.423}\log^2{x}+ 2.074\sqrt{x}\log{x}+14.712x^{0.423}\log{x}\\
&\quad\quad +18.610x^{0.423}\log\log{(qx)}+(2.382\log q+8.018)x^{0.423}\\
&\quad\quad+\left(127.562\left(\log\log{(qx)}\right)^2+\left(32.449\log{q}-1.720\right)\log\log{(qx)} \right.\\
&\quad\quad\left.+2.064\log^2q+6.570\log q+79.146\right)\frac{x^{0.423}}{\log {x}}\\
&\quad\quad+13.962\frac{\log{x}}{x^{0.077}}+\frac{8.4}{x^{0.077}}\log{\log{(qx)}}+\frac{1.255\log{q}+8.510}{x^{0.077}}.
\end{align*}
\end{theorem}

\begin{proof}
The main idea of the proof is following: By Lemma \ref{Jest} we can estimate the function $\psi_0(x,\chi)$ with the function $J(x,T+1,\chi)$. We will then further estimate the function $J(x,T+1,\chi)$. This is done by estimating this function using suitable integrals of the function $\frac{L'(s,\chi)}{L(s,\chi)}\frac{x^s}{s}$. Finally, putting together these bounds will give the desired result. Let us now continue to the details of the proof.

\begin{figure}[h]
\begin{tikzpicture}
\draw[->] (-7,0)--(3,0);
\draw[->] (0,-2)--(0,4);

\draw (2.2,-1) -- (2.2,3) -- (-6,3)--(-6,0.2);
\draw (-6,-0.2) -- (-6,-1) -- (2.2,-1);
\draw [domain=90:270] plot ({-6+0.2*cos(\x)}, {0.2*sin(\x)});

\node[] at (2.5, -1.5)  {$c+iT_2$};
\node[] at (2.5, 3.3)  {$c+iT_1$};
\node[] at (-6.3, 3.3)  {$-U+iT_1$};
\node[] at (-6.3, -1.5)  {$-U+iT_2$};
\end{tikzpicture}
\caption{$\mathcal{R}$ with a modified left line}
\label{RFigure}
\end{figure}
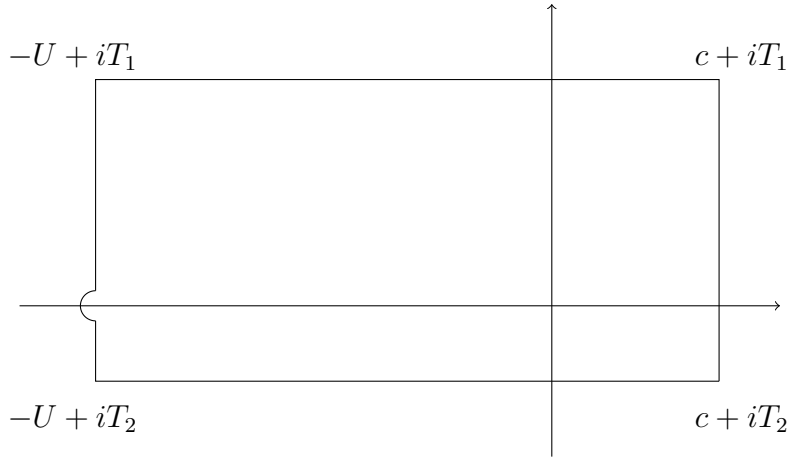

We estimate the function $J(x,T+1,\chi)$ with an integral over a (modified) rectangle and estimate necessary other integrals. Since we want to avoid the poles of the function $\frac{L'(s,\chi)}{L(s,\chi)}\frac{x^s}{s}$, we select the horizontal lines of the rectangle carefully. Let $T_1$ and $T_2$ be as in Lemma \ref{lemmaDifference} with respect to $T$ (which means that $T_1\in (T-1,T+1]$ and $T_2 \in [-T-1,-T+1)$). There does not exists zeros $\rho$ with $\Im\rho \in \{T_1,T_2\}$ and we set the horizontal lines of the rectangle at $y=T_1$ and $y=T_2$. Next we define vertical lines of the rectangle. Let $U>1$. If there does not exist a zero $\rho=-U$, then $\mathcal{R}$ is a rectangle with vertices 
\begin{equation*}
c+iT_2 \quad c+iT_1 \quad -U+iT_1\quad \text{and} \quad-U+iT_2,
\end{equation*} 
where $c=1+\frac{1}{\log{x}}$. Otherwise we avoid the point $U$ with a circle which has a small radius and a circumcentre at $(-U,0)$ (see Figure \ref{RFigure}). We denote the left horizontal part of $\mathcal{R}$ with $\mathcal{R}_1$. We have
\begin{equation}
\label{LxIntegral}
\begin{aligned}
&J(x,T+1,\chi) \\
&\quad= \frac{1}{2\pi i}\int_{c-i(T+1)}^{c+i(T+1)} \left(-\frac{L'(s,\chi)}{L(s,\chi)}\right)\frac{x^s}{s} ds\\
&\quad=\frac{1}{2\pi i}\int_{c-i(T+1)}^{c+iT_2} \left(-\frac{L'(s,\chi)}{L(s,\chi)}\right)\frac{x^s}{s} ds+\frac{1}{2\pi i}\int_{c+iT_1}^{c+i(T+1)} \left(-\frac{L'(s,\chi)}{L(s,\chi)}\right)\frac{x^s}{s} ds  \\
& \quad\quad+\frac{1}{2\pi i}\int_\mathcal{R} \left(-\frac{L'(s,\chi)}{L(s,\chi)}\right)\frac{x^s}{s} ds-\frac{1}{2\pi i}\int_{\mathcal{R}_1} \left(-\frac{L'(s,\chi)}{L(s,\chi)}\right)\frac{x^s}{s} ds \\
& \quad\quad-\frac{1}{2\pi i}\int_{c+iT_1}^{-U+iT_1} \left(-\frac{L'(s,\chi)}{L(s,\chi)}\right)\frac{x^s}{s} ds-\frac{1}{2\pi i}\int_{-U+iT_2}^{c+iT_2} \left(-\frac{L'(s,\chi)}{L(s,\chi)}\right)\frac{x^s}{s} ds.
\end{aligned}
\end{equation}
The goal is to estimate the right hand side of the previous formula.

First we consider the first line on the right hand side of the previous formula. By Lemma \ref{L2tEst}
\begin{equation}
\label{firstLine}
\begin{aligned}
& \left|\frac{1}{2\pi i}\int_{c-i(T+1)}^{c+iT_2} \left(-\frac{L'(s,\chi)}{L(s,\chi)}\right)\frac{x^s}{s} ds+\frac{1}{2\pi i}\int_{c+iT_1}^{c+i(T+1)} \left(-\frac{L'(s,\chi)}{L(s,\chi)}\right)\frac{x^s}{s} ds\right| \\
& \quad<\frac{xe}{2\pi}\left(\log{x}+\gamma+\frac{0.478}{\log{x}}\right)\left(\int_{c-i(T+1)}^{c+iT_2} \frac{1}{|s|} ds+\int_{c+iT_1}^{c+i(T+1)} \frac{1}{|s|} ds\right)\\
& \quad \le\left(\log{x}+\gamma+\frac{0.478}{\log{x}}\right)\frac{2xe}{\pi(T-1)}.
\end{aligned}
\end{equation}
Thus it is sufficient to estimate the last two lines of formula \eqref{LxIntegral}.

Next we consider the second line on the right hand side of formula \eqref{LxIntegral}. We notice that
\begin{align*}
& \frac{1}{2\pi i}\int_\mathcal{R} \left(-\frac{L'(s,\chi)}{L(s,\chi)}\right)\frac{x^s}{s} ds \\
& \quad= -\sum_{\substack{\rho \\ T_2<\Im\rho<T_1}} \frac{x^\rho}{\rho}-\mathfrak{a}\frac{L'(0,\chi)}{L(0,\chi)}-(1-\mathfrak{a})\left(\log{x}+b(\chi)\right)+\sum_{1\leq 2m-\mathfrak{a}\leq U}\frac{x^{\mathfrak{a}-2m}}{2m-\mathfrak{a}}.
\end{align*}
Further, by \cite[Pages 116--117]{davenport} we have $\frac{L'(s,\chi)}{L(s,\chi)}=O(\log{(q|s|)})$ for $\Re s \leq -1$. Thus
\begin{equation*}
\frac{1}{2\pi i}\int_{\mathcal{R}_1} \left(-\frac{L'(s,\chi)}{L(s,\chi)}\right)\frac{x^s}{s} ds \ll \frac{T\log{U}}{Ux^U}.
\end{equation*}
Hence, when $U$ goes to infinity, the second line on the right hand side of formula \eqref{LxIntegral} is
\begin{equation}
\label{secondLine}
-\sum_{\substack{\rho \\ T_2<\Im\rho<T_1}} \frac{x^\rho}{\rho}-\mathfrak{a}\frac{L'(0,\chi)}{L(0,\chi)}-(1-\mathfrak{a})\left(\log{x}+b(\chi)\right)+\sum_{m=1}^\infty\frac{x^{\mathfrak{a}-2m}}{2m-\mathfrak{a}}.
\end{equation}
Since it is easier to estimate the function $\psi_0(x,\chi)$ using the term $\sum_{|\Im\rho|\leq T+1}\frac{x^\rho}{\rho}$, we notice that by estimate \eqref{estimateNZeros} and assuming the GRH we have
\begin{equation}
\label{xRhoSum}
\begin{aligned}
\left|\sum_{\substack{\rho \\ T_2<\Im\rho<T_1}} \frac{x^\rho}{\rho}-\sum_{\substack{\rho \\ |\Im\rho|\leq T+1}} \frac{x^\rho}{\rho}\right|&=\left|\sum_{\substack{\rho \\  T_1\le\Im\rho\leq T+1\\-T-1\leq \Im\rho\leq T_2}}\frac{x^\rho}{\rho}\right|\\ 
& \leq \frac{\sqrt{x}}{T-1}\left(1.092\log{(qT)}+4\log{\log{(qT)}}-1.250\right).
\end{aligned}
\end{equation}
Thus we have estimated the second line on the right hand side of formula \eqref{LxIntegral}. 

Next we consider the third line on the right hand side of formula \eqref{LxIntegral}. We divide the investigation to two cases: first we estimate the integrals with $-1\leq \Re s\leq c$ and then the integrals with $-U\leq \Re s <-1$. By Lemma \ref{est1c} 
\begin{equation}
\label{thirdLinefirst}
\begin{aligned}
& \left|-\frac{1}{2\pi i}\int_{c+iT_1}^{-1+iT_1} \left(-\frac{L'(s,\chi)}{L(s,\chi)}\right)\frac{x^s}{s} ds-\frac{1}{2\pi i}\int_{-1+iT_2}^{c+iT_2} \left(-\frac{L'(s,\chi)}{L(s,\chi)}\right)\frac{x^s}{s} \right| \\
& \quad<\frac{1}{\pi (T-1)}R(T,\chi)\int_{-1}^{c} x^\sigma d\sigma \\
& \quad <\frac{xe}{\pi (T-1)\log{x}}R(T,\chi),
\end{aligned}
\end{equation}
where
\begin{align*}
R(T,\chi)&= 2.385\log^2{\left(q(T+1)\right)}+17.472\log{\left(q(T+1)\right)}\log{\log{(q(T+1))}} \\
&\quad-3.276\log{\left(T+1\right)}+\frac{13}{8}\log{\left(\frac{T^2}{4}+\frac{T}{2}+\frac{5}{2}\right)} +32\left(\log{\log{(q(T+1))}}\right)^2\\
&\quad-12\log{\log{(q(T+1))}}+7.592\log q+79.427+\frac{3\pi}{4(T-1)}+\frac{3}{(T-1)^2}.
\end{align*}

Thus we have estimated the case $-1\le\Re s\leq c$ and it is sufficient to estimate the integrals for $\Re(s)<-1$. Since $T>3$, by Lemma \ref{LN1} we have
\begin{align*}
& \left|-\frac{1}{2\pi i}\int_{-1+iT_1}^{-U+iT_1} \left(-\frac{L'(s,\chi)}{L(s,\chi)}\right)\frac{x^s}{s} ds-\frac{1}{2\pi i}\int_{-U+iT_2}^{-1+iT_2} \left(-\frac{L'(s,\chi)}{L(s,\chi)}\right)\frac{x^s}{s} \right| \\
& \quad<\frac{1}{2\pi}\left|\int_{-1+iT_1}^{-U+iT_1}  \left(\left|\log{\frac{q}{\pi}}\right|+ \frac{3}{2}\log{\left(\frac{|1-s+\mathfrak{a}|}{2}\right)} \right.\right. \\
& \quad\quad\left.\left.+\frac{3}{2}\log{\left(\frac{|s+\mathfrak{a}|}{2}\right)}+3.570+\gamma+\frac{8}{T-1}\right)\frac{x^\sigma}{|s|} d\sigma\right| \\
& \quad\quad+\frac{1}{2\pi}\left|\int_{-1+iT_2}^{-U+iT_2}  \left(\left|\log{\frac{q}{\pi}}\right|+ \frac{3}{2}\log{\left(\frac{|1-s+\mathfrak{a}|}{2}\right)}\right.\right. \\
& \quad\quad\left.\left.+\frac{3}{2}\log{\left(\frac{|s+\mathfrak{a}|}{2}\right)}+3.570+\gamma+\frac{8}{T-1}\right)\frac{x^\sigma}{|s|} d\sigma\right|.
\end{align*}
Since $|1-s+\mathfrak{a}| \leq |s|+2$ and $|s+\mathfrak{a}|\leq |s|<|s|+2$ and the function  $\frac{\log{\left(\frac{|s|+2}{2}\right)}}{|s|}$  is decreasing for $|s|$, the right hand side of the previous formula is
\begin{equation}
\label{thirdLinesecond}
\begin{aligned}
& <\left(\frac{3\log{\left(\frac{T+1}{2}\right)}+\left|\log{\frac{q}{\pi}}\right|+3.570+\gamma}{T-1}+\frac{8}{(T-1)^2}\right) \frac{1}{\pi}\int_{-U}^{-1} x^\sigma d\sigma \\
& \quad < \frac{3\log{\left(\frac{T+1}{2}\right)}+\log{(q\pi)}+3.570+\gamma+\frac{8}{T-1}}{\pi (T-1)x\log{x}}.
\end{aligned}
\end{equation}
Thus we have estimated the term $J(x,T+1,\chi)$ and we estimate the function $\psi_0(x,\chi)$.

Hence, by Lemma \ref{Jest} and formulas \eqref{LxIntegral},\eqref{firstLine}, \eqref{secondLine}, \eqref{xRhoSum}, \eqref{thirdLinefirst} and \eqref{thirdLinesecond} we have
\begin{align*}
 & \left|\psi_0(x,\chi)+\sum_{\substack{\rho \\ |\Im\rho|\leq T+1}} \frac{x^\rho}{\rho}+\mathfrak{a}\frac{L'(0,\chi)}{L(0,\chi)}+(1-\mathfrak{a})\left(\log{x}+b(\chi)\right)-\sum_{m=1}^\infty\frac{x^{\mathfrak{a}-2m}}{2m-\mathfrak{a}}\right| \\
&<1.363\frac{x\log^2 x}{T+1} +2.074\sqrt{x}\log x+12.294\frac{x\log x}{T+1}+\frac{2xe}{\pi(T-1)}\log{x}\\
&\quad+\frac{2.385xe}{\pi (T-1)\log{x}}\log^2{\left(T+1\right)}+\frac{17.472xe}{\pi (T-1)\log{x}}\log{\left(T+1\right)}\log{\log{(q(T+1))}}  \\
& \quad+\left(\frac{7.032}{T+1}+\frac{2e\gamma}{\pi(T-1)}\right)x +\frac{xe}{\pi (T-1)\log{x}}\left(4.77\log{q}\log{(T+1)} \vphantom{\frac{13}{8}\log{\left(\frac{T^2}{4}+\frac{T}{2}+\frac{5}{2}\right)}}\right. \\
&\quad\left.-3.276\log{\left(T+1\right)}+\frac{13}{8}\log{\left(\frac{T^2}{4}+\frac{T}{2}+\frac{5}{2}\right)} \right)\\ 
&\quad +\frac{xe}{\pi (T-1)\log{x}}\left(32\left(\log{\log{(q(T+1))}}\right)^2 \right. \\
&\quad\left.+\left(17.472\log{q}-12\right)\log{\log{(q(T+1))}}\vphantom{32\left(\log{\log{(q(T+1))}}\right)^2}\right)\\
& \quad+\frac{5.823x}{(T+1)\log x} +\left(0.478\cdot2+79.427\right)\frac{x e}{\pi (T-1)\log{x}} \\
&\quad+\frac{xe}{\pi (T-1)\log{x}}\left(2.385\log^2{q}+7.592\log{q}\right)\\
& \quad+12.624\frac{\sqrt{x}\log x}{T+1}+\frac{1.092\sqrt{x}}{T-1}\log{T}+\frac{4\sqrt{x}}{T-1}\log{\log{(qT)}}+\frac{\log x}{T+1} \\
&\quad+\frac{1.092\sqrt{x}}{T-1}\log{q}+0.893\frac{\sqrt{x}}{T+1}-1.250\frac{\sqrt{x}}{T-1}+\frac{1}{T+1} \\
& \quad+\frac{xe}{\pi (T-1)\log{x}}\left(\frac{3\pi}{4(T-1)}+\frac{3}{(T-1)^2}\right) \\
&\quad+ \frac{3\log{\left(\frac{T+1}{2}\right)}+\log{(q\pi)}+3.570+\gamma+\frac{8}{T-1}}{\pi (T-1)x\log{x}}.
\end{align*}
We bound this expression for $x\geq 2$ using the value $T=x^{0.577}+8.509$. For the first term, the sixth and the seventh line on the right hand side of the previous inequality we can estimate $T+1, T-1>x^{0.577}$. The claim follows from these estimates and Lemmas \ref{lemma423log} to \ref{lemmaasympt0}.
\end{proof}

By the previous result we can limit to estimating the contribution which comes from computing the sums up to height $T$.

\begin{lemma} 
\label{smallRho} 
Assume GRH for the function $L(s,\chi)$ and that $\chi$ is a primitive non-principal character. Let $x \geq 2$ be a real number and let $T=x^{0.577}+8.509$.  Then
\begin{multline*}
\left|\sum_{|\Im \rho |\leq T+1}\frac{x^{\rho}}{\rho}\right| <\frac{1}{6\pi}\sqrt{x}\log^2 x +(0.184\log q-0.337)\sqrt{x}\log x \\
+(1.693\log q+11.946)\sqrt{x}
\end{multline*}
when $\rho$ runs through zeros of $L(s, \chi)$ up to height $T+1$.
\end{lemma}
\begin{proof}
First we consider the zeros which absolute values of the imaginary parts are at most $5/7$ and then we consider the rest of the zeros. Assuming the GRH we have
\begin{equation*}
\left|\sum_{|\Im \rho |\leq 5/7}\frac{x^{\rho}}{\rho}\right|\leq 2\sum_{|\Im \rho |\leq 5/7}x^{1/2}\leq 2N(5/7,\chi)\sqrt{x}.
\end{equation*}

Let us now consider the case with $5/7<|\Im\rho|\leq T+1$ using partial summation. We have $\left|\sum_{5/7<|\Im \rho|\leq T+1} \frac{x^{\rho}}{\rho}\right|\leq \sqrt{x}\sum_{5/7<|\Im \rho|\leq T+1}\frac{1}{|\Im \rho|}$. Write $\rho=\frac{1}{2}\pm i\gamma$, where $\gamma$ is non-negative. By Theorem \ref{NZeros2}, we have 
\begin{align*} 
\sum_{5/7<|\gamma|\leq T+1}\frac{1}{\gamma}&=\frac{1}{T+1}\left(N(T+1,\chi)-N(5/7,\chi)\right)\\
&\quad +\int_{5/7}^{T+1}\frac{1}{t^2}\left(N(t,\chi)-N(5/7,\chi)\right)dt\\
&=\frac{1}{T+1}N(T+1,\chi)-\frac{7N(5/7,\chi)}{5}+\int_{5/7}^{T+1}\frac{1}{t^2}N(t,\chi)dt \\
&\leq \frac{1}{T+1}\left(\frac{T+1}{\pi}\log \frac{q(T+1)}{2\pi e}+0.247\log{\frac{q(T+1)}{2\pi}}+6.894\right)\\
&\quad-\frac{7N(5/7,\chi)}{5}+\int_{5/7}^{T+1}\frac{1}{t^2}\left(\frac{t}{\pi}\log \frac{qt}{2\pi e}+0.247\log{\frac{qt}{2\pi}}+6.894\right)dt \\ 
&=\frac{1}{\pi}\log \frac{q(T+1)}{2\pi e}+\frac{0.247}{T+1}\log{\frac{q(T+1)}{2\pi}}+\frac{6.894}{T+1}\\
&\quad-\frac{7N(5/7,\chi)}{5}+\frac{1}{2\pi}\left(\log^2{(T+1)}-\log^2{\left(\frac{5}{7}\right)}\right)\\
&\quad+\frac{1}{\pi}\left(\log{(T+1)}-\log{(5/7)}\right)\log{\frac{q}{2\pi e}}  \\
&\quad+ 0.247\left(-\frac{\log{(T+1)}+1}{T+1}+\frac{\log{(5/7)}+1}{5/7}\right) \\
&\quad+0.247\left(\frac{7}{5}-\frac{1}{T+1}\right)\log \frac{q}{2\pi e}-\frac{6.894}{T+1}+\frac{7\cdot6.894}{5}.
\end{align*}

Combining everything and using Theorem \ref{NZeros2}, we get 
\begin{align*}
&\left|\sum_{|\Im \rho |\leq T+1}\frac{x^{\rho}}{\rho}\right|=\left|\sum_{5/7<|\Im \rho |\leq T+1}\frac{x^{\rho}}{\rho}\right|+\left|\sum_{|\Im \rho |\leq 5/7}\frac{x^{\rho}}{\rho}\right| \\ 
&\quad\leq \sqrt{x}\left(\frac{1}{\pi}\log \frac{q(T+1)}{2\pi e}+\frac{0.247}{T+1}\log{\frac{q(T+1)}{2\pi}} \right.\\
&\quad\quad\left.+\frac{1}{2\pi}\left(\log^2{(T+1)}-\log^2{\left(\frac{5}{7}\right)}\right)+\frac{1}{\pi}\left(\log{(T+1)}-\log{(5/7)}\right)\log{\frac{q}{2\pi e}}  \right. \\
&\quad\quad\left.+0.247\left(-\frac{\log{(T+1)}+1}{T+1}+\frac{\log{(5/7)}+1}{5/7}\right) \right.\\
&\quad\quad\left.+0.247\left(\frac{7}{5}-\frac{1}{T+1}\right)\log \frac{q}{2\pi e}+\frac{48.258}{5} \right) \\
&\quad\quad+\frac{3}{5}\left(\frac{5}{7\pi}\log \frac{5q}{14\pi e}+0.247\log{\frac{5q}{14\pi}}+6.894\right)\sqrt{x}\\ 
&\quad=\left( \frac{\log^2(T+1)}{2\pi}+\left(\frac{1}{\pi}\log{q}-0.585\right)\log(T+1)\right)\sqrt{x}\\
&\quad\quad+\left(\left(\frac{1}{\pi}-\frac{1}{\pi}\log{(5/7)}+0.247\cdot\frac{7}{5}+\frac{3}{5}\cdot\frac{5}{7\pi}+\frac{3}{5}\cdot0.247\right)\log q \right.\\
&\quad\quad\left.-\frac{\log{(2\pi e)}}{\pi}-\frac{1}{2\pi}\log^2{\left(\frac{5}{7}\right)}+\frac{\log{(2\pi e)}}{\pi}\log{(5/7)}+0.247\cdot\frac{\log{(5/7)}+1}{5/7} \right. \\
&\quad\quad\left.-0.247\cdot\frac{7}{5}\log{(2\pi e)}+\frac{48.258}{5} \right. \\
&\quad\quad\left.+ \frac{3}{5}\left(\frac{5}{7\pi}\log \frac{5}{14\pi e}+0.247\log{\frac{5}{14\pi}}+6.894\right) \right)\sqrt{x} \\
&\quad<\left( \frac{\log^2(T+1)}{2\pi}+\left(\frac{1}{\pi}\log{q}-0.585\right)\log(T+1)+1.056\log q+11.056\right)\sqrt{x}.
\end{align*}
Since the functions $\log^2(T+1)-\frac{\log^2 x}{3}$ and $\log(T+1)-\log x$, where $T=x^{0.577}+8.509$, are decreasing for $x \geq 2$, we have $\log^2(T+1)<\frac{\log^2 x}{3}+5.591$ and $0.577\log x<\log(T+1)<0.577\log x+1.999$. Using these estimates, we obtain the bound
\begin{align*}
&\left|\sum_{|\Im \rho |\leq T+1}\frac{x^{\rho}}{\rho}\right|<\left( \frac{\frac{\log^2 x}{3}+5.591}{2\pi}+\frac{1}{\pi}\log{q}\left(0.577\log x+1.999\right)-0.585\cdot0.577\log x \right. \\
& \quad\left.+1.056\log q+11.056 \vphantom{\frac{\frac{\log^2 x}{3}+5.591}{2\pi}+\frac{1}{\pi}\log{q}}\right)\sqrt{x} \\
&\quad <\frac{1}{6\pi}\sqrt{x}\log^2 x +(0.184\log q-0.337)\sqrt{x}\log x +(1.693\log q+11.946)\sqrt{x}.
\end{align*}
\end{proof}

\subsection{Proof of Theorem \ref{psi}}

In this section we prove an estimate for the function $\psi(x;q,a)$. Remember, that in Theorem \ref{psi} this estimate is described in the form $\left|\psi(x;q,a)-x/\varphi(q)\right|<\ldots$. 

\begin{proof}[Proof of Theorem \ref{psi}]
By formula \eqref{psiFormula} we have
\begin{equation*}
\psi(x;q,a)=\frac{1}{\varphi(q)}\sum_{\chi}\overline{\chi}(a)\psi(x,\chi)=\frac{\psi(x)}{\varphi(q)}+\frac{c_1\log x\log q}{\varphi(q)}+\frac{1}{\varphi(q)}\sum_{\chi\ne \chi_0}\overline{\chi}(a)\psi(x,\chi),
\end{equation*}
where $c_1$ satisfies $-\frac{2}{\log 6}\leq c_1\leq 0$ and $\gcd(a,q)=1$. It is sufficient to estimate the first and the third term on the right hand side of previous formula.

Since we assume GRH$(q)$ for all moduli dividing $q$, in particular, RH holds. Therefore, using Theorem 10 from the article \cite{schoenfeld}, we have $|\psi(x)-x|<\frac{1}{8\pi}\sqrt{x}\log^2x$ whenever $x>73.2$. It is a straightforward calculation to perform with Sage \cite{sage} to verify that the bound 
\begin{equation*}
|\psi(x)-x|<\frac{1}{8\pi}\sqrt{x}\log^2x+2.6
\end{equation*} 
holds for $2\leq x\leq 74$. This is done by checking that for any integer on the interval $[1,74]$, the difference satisfies the inequality $|\psi(x)-x|<\frac{1}{8\pi}\sqrt{x}\log^2x+1.6$, and this bound is then extended to reals by bounding in the following way:
\begin{align*}
|\psi(x)-x|&=|\psi(\lfloor x\rfloor)-x|\leq \left|\psi(\lfloor x\rfloor)-\lfloor x\rfloor\right| +1\\
&<2.6+\frac{1}{8\pi}\sqrt{\lfloor x\rfloor }\log^2\lfloor x\rfloor \\
&\leq \frac{1}{8\pi}\sqrt{x}\log^2x+2.6.
\end{align*}
 This takes care of the first term on the right-hand side.

Finally, we estimate the term $\frac{1}{\varphi(q)}\sum_{\chi\ne \chi_0}\overline{\chi}(a)\psi(x,\chi)$. We have
\begin{equation*}
\left|\frac{1}{\varphi(q)}\sum_{\chi\ne \chi_0}\overline{\chi}(a)\psi(x,\chi)\right|\le\frac{1}{\varphi(q)}\sum_{\chi\ne \chi_0}\left|\psi(x,\chi)\right|.
\end{equation*}
By formula \eqref{formulaPsi0} the term $\psi(x,\chi)$ can be estimated by the term $\psi_0(x,\chi)$ with an error of at most $\frac{1}{2}\log x$ per character. Furthermore,
Lemma \ref{psiNonPrimitive} allows us to replace each $\psi(x,\chi)$ by $\psi(x,\chi^\star)$, where $\chi^\star$ is the primitive character that induces $\chi$, introducing an error of at most $2\log x$ if $q=6$ and $\log{q}\log x$ otherwise. In total, remembering the contribution $\frac{1}{2}\log x$, the error caused by these changes for $q\ge3$ will be at most
\begin{equation*}
\begin{cases}
2\log{x}+\frac{1}{2}\log{x}=2.5\log{x}<1.396\log 6\log{x} \quad&\text{if}\quad q=6 \\
\log{q}\log{x}+\frac{1}{2}\log{x}=\left(\log{q}+0.5\right)\log{x}< 1.456\log{q}\log{x} \quad&\text{otherwise}.
\end{cases}
\end{equation*}
Later we use the upper bound $1.456\log{q}\log{x}$. 

By the previous paragraph, it is sufficient to consider only primitive non-principal characters $\chi^{\star}$. By formulas \eqref{psi0M1} and \eqref{psi01} we have
\begin{equation}
\label{psiTheo1}
\psi_0(x,\chi^{\star})=-\sum\limits_\rho\frac{x^\rho}{\rho}-\mathfrak{a}\frac{L'(0,\chi^{\star})}{L(0,\chi^{\star})}-(1-\mathfrak{a})(\log{x}+b(\chi^{\star}))+\sum\limits_{m=1}^\infty\frac{x^{\mathfrak{a}-2m}}{2m-\mathfrak{a}},
\end{equation}
where $\mathfrak{a}=0$ if $\chi^{\star}(-1)=1$ and $\mathfrak{a}=1$ if $\chi^{\star}(-1)=-1$. Further, let us remember that $\chi^{\star}(-1)=\chi(-1)$ when a character $\chi^{\star}$ induces a character $\chi$ which is of modulus $q$. Half of the characters $\chi$ have $\mathfrak{a}=0$, including the principal character, and the other half have $\mathfrak{a}=1$. Hence, at most half of the characters $\chi^{\star}$ have $\mathfrak{a}=0$ and the other half have $\mathfrak{a}=1$ even if the primitive characters may be of different modulus. Further, our estimates were done as following: each term $\psi_0(x,\chi)$ with non-principal character $\chi$ modulo $q$ was estimated by term $\psi_0(x,\chi^{\star})$ where a character $\chi^{\star}$ induces a character $\chi$ which is of modulus $q$. Hence, the number of primitive non-principal characters $\chi^{\star}$ under the consideration equals to the number of non-principal characters $\chi$ of modulus $q$. It is $\varphi(q)-1$. Summing over the values of the characters $\chi^{\star}$ and dividing by their number, the last two terms contribute at most
\begin{equation*}
\frac{1}{2}\left(\log x+2.751\log{q}+23.878\right)+1
\end{equation*}
by Lemmas \ref{Lemmabchi} and \ref{sumx}. 

Then we may move to the contribution coming from the logarithmic derivative of the $L$-functions at zero. Clearly, this contribution may be written as
\begin{equation*}
\frac{1}{\varphi(q)}\sum_{\chi^{\star}(-1)=-1}\left|\frac{L'(0,\chi^{\star})}{L(0,\chi^{\star})}\right|.
\end{equation*}
This can be estimated by Lemma \ref{Lremark} and multiplying by $\frac{\varphi(q)}{2}$. 

We may now move forward to considering the first term on the right hand side of formula \eqref{psiTheo1}. We use Theorem \ref{largeRho} to cut the sum over the zeros and then Lemma \ref{smallRho} to bound the sum over the zeros of the bounded height. We notice again that we may need to consider characters of different modulus. Since the bounds given by Theorem \ref{largeRho} and Lemma \ref{smallRho} are increasing with respect to the modulus, the worst bound is given by using the original modulus $q$. Putting everything together, we have 
\begin{align*}
&\left|\psi(x;q,a)-\frac{x}{\varphi(q)}\right|<\frac{2\log x\log q}{\varphi(q)\log{6}}+\frac{\frac{1}{8\pi}\sqrt{x}\log^2x+2.6}{\varphi(q)}+1.456\log{q}\log{x}\\ 
&\quad +\frac{1}{2}\left(\log x+2.751\log{q}+23.878\right)+1+\frac{1}{\varphi(q)}\sum_{\chi^{\star}(-1)=-1}\left|\frac{L'(0,\chi^{\star})}{L(0,\chi^{\star})}\right| \\
&\quad +1.363x^{0.423}\log^2{x}+ 2.074\sqrt{x}\log{x}+14.712x^{0.423}\log{x}\\
&\quad+18.610x^{0.423}\log\log{(qx)}+(2.382\log q+8.018)x^{0.423}\\
&\quad+\left(127.562\left(\log\log{(qx)}\right)^2+\left(32.449\log{q}-1.720\right)\log\log{(qx)} \right.\\
&\quad\left.+2.064\log^2q+6.570\log q+79.146\right)\frac{x^{0.423}}{\log {x}}\\
&\quad+13.962\frac{\log{x}}{x^{0.077}}+\frac{8.4}{x^{0.077}}\log{\log{(qx)}}+\frac{1.255\log{q}+8.510}{x^{0.077}} \\
&\quad+\frac{1}{6\pi}\sqrt{x}\log^2 x +(0.184\log q-0.337)\sqrt{x}\log x +(1.693\log q+11.946)\sqrt{x}.
\end{align*}
We may combine together similar terms to obtain
\begin{align*}
& < \frac{4\varphi(q)+3}{24 \pi\varphi(q)}\sqrt{x}\log^2 x+ \left(0.184\log q+1.737\right)\sqrt{x}\log x+1.363x^{0.423}\log^2{x}\\
&\quad +(1.693\log q+11.946)\sqrt{x}+14.712x^{0.423}\log{x}\\
&\quad+18.610x^{0.423}\log\log{(qx)}+(2.382\log q+8.018)x^{0.423} +\left(127.562\left(\log\log{(qx)}\right)^2 \right.\\
&\quad\left.+\left(32.449\log{q}-1.720\right)\log\log{(qx)}+2.064\log^2q+6.570\log q+79.146\right)\frac{x^{0.423}}{\log {x}}\\ 
& \quad +\frac{1}{\varphi(q)}\sum_{\chi^{\star}(-1)=-1}\left|\frac{L'(0,\chi^{\star})}{L(0,\chi^{\star})}\right|+\left(\frac{2\log q}{\varphi(q)\log{6}}+1.456\log q+0.5\right)\log x \\
&\quad+1.376\log q+12.939+\frac{2.6}{\varphi(q)}+13.962\frac{\log{x}}{x^{0.077}}+\frac{8.4}{x^{0.077}}\log{\log{(qx)}} \\
&\quad+\frac{1.255\log{q}+8.510}{x^{0.077}}.
\end{align*}
Further, since $x^{0.423}<0.949\sqrt{x}$, $x^{0.0385}<0.974x^{0.077}$, $\log{\log{(qx)}}\geq 0.583$ and $\varphi(q) \geq 2$ for $x\geq 2$, $q \geq 3$, using Lemma \ref{lemmapsiInequalities} this can be written as
\begin{align*}
&< \frac{4\varphi(q)+3}{24 \pi\varphi(q)}\sqrt{x}\log^2 x+ \left(0.184\log q+1.737+1.363\cdot 4.778\right)\sqrt{x}\log x\\
&\quad+\left(1.693 + 18.610\cdot0.077\cdot0.949 + 2.382\cdot0.949\right)\log q\sqrt{x}\\
&\quad+\left(11.946+14.712\cdot4.778 \right.\\
&\quad\left.+18.61\cdot(4.778\cdot0.077+0.949\cdot\log{4.778})+8.018\cdot0.949\right)\sqrt{x} \\
&\quad+\left(127.562(0.077^2\cdot0.949)+32.449\cdot0.077\cdot0.949+2.064\cdot0.949\right)\log^2q\frac{\sqrt{x}}{\log {x}} \\
&\quad+\left(127.562\cdot2\left(0.077^2\cdot9.556\cdot0.974+0.077\cdot\log{(4.778)}\cdot0.949\right)\right. \\
&\quad\left.+32.449\left(4.778\cdot0.077+0.949\cdot\log{4.778}\right)+6.57\cdot0.949\right)\log q\frac{\sqrt{x}}{\log {x}} \\
&\quad+\left(\left(79.146-1.72\cdot0.583\right)\cdot0.949 +127.562\left((0.077\cdot9.556)^2\right.\right. \\
&\quad\left.\left.+\log^2{(4.778)}\cdot0.949+2\cdot0.077\cdot9.556\cdot\log{(4.778)}\cdot0.974\right)\right)\frac{\sqrt{x}}{\log {x}}\\ 
& \quad +\frac{1}{\varphi(q)}\sum_{\chi^{\star}(-1)=-1}\left|\frac{L'(0,\chi^{\star})}{L(0,\chi^{\star})}\right|+\left(\frac{\log q}{\log{6}}+1.456\log q+0.5\right)\log x \\
&\quad+1.376\log q+12.939+2.6/2+13.962\cdot4.778+8.4\cdot4.778\cdot0.077 \\
&\quad+\frac{(8.4\cdot0.077+1.255)\log{q}+8.4 \log{4.778} + 8.510}{2^{0.077}}
\end{align*}
and simplified to
\begin{align*}
&< \frac{4\varphi(q)+3}{24 \pi\varphi(q)}\sqrt{x}\log^2 x+ \left(0.184\log q+8.250\right)\sqrt{x}\log x\\
&\quad+(5.314\log q+124.318)\sqrt{x}+(5.048\log^2q+109.573\log q+725.316)\frac{\sqrt{x}}{\log {x}}\\ 
& \quad +\frac{1}{\varphi(q)}\sum_{\chi^{\star}(-1)=-1}\left|\frac{L'(0,\chi^{\star})}{L(0,\chi^{\star})}\right|+\left(2.015\log q+0.5\right)\log x+3.179\log q+104.563.
\end{align*}

Furthermore, we would like to simplify this even more and thus remove the logarithmic derivatives of the $L$-function. By Lemma \ref{Lremark}, the right hand side of the previous inequality is
\begin{align*}
&< \frac{4\varphi(q)+3}{24 \pi\varphi(q)}\sqrt{x}\log^2 x+ \left(0.184\log q+8.250\right)\sqrt{x}\log x+(5.314\log q+124.318)\sqrt{x}\\
&\quad+(5.048\log^2q+109.573\log q+725.316)\frac{\sqrt{x}}{\log {x}}+\left(2.015\log q+0.5\right)\log x\\ 
& \quad +0.014\sqrt{q}\log{q}+0.034\sqrt{q}+3.679\log q+263.886
\end{align*}
if $3\leq q<4\cdot 10^5$,
\begin{align*}
&< \frac{4\varphi(q)+3}{24 \pi\varphi(q)}\sqrt{x}\log^2 x+ \left(0.184\log q+8.250\right)\sqrt{x}\log x\\
&\quad+(5.314\log q+124.318)\sqrt{x}+(5.048\log^2q+109.573\log q+725.316)\frac{\sqrt{x}}{\log {x}}\\ 
& \quad +\left(2.015\log q+0.5\right)\log x+1.858\log^2{q}+3.679\log q+104.626
\end{align*}
if $4\cdot 10^5\leq q<10^{29}$ and
\begin{align*}
&< \frac{4\varphi(q)+3}{24 \pi\varphi(q)}\sqrt{x}\log^2 x+ \left(0.184\log q+8.250\right)\sqrt{x}\log x\\
&\quad+(5.314\log q+124.318)\sqrt{x}+(5.048\log^2q+109.573\log q+725.316)\frac{\sqrt{x}}{\log {x}}\\ 
& \quad +\left(2.015\log q+0.5\right)\log x+(0.297\log\log q+0.603)\log^2{q} +3.679\log q+104.626
\end{align*}
if $q \geq 10^{29}$.
\end{proof}

\section{Proof of Theorem \ref{pi}}
\label{secProofPi}

In this section we prove the estimate described in Theorem \ref{pi} for the function $\psi(x;q,a)$. Before that, we prove some preliminary estimates which will be used in the proof.

\subsection{Preliminary estimates for the proof of Theorem \ref{pi}}
\label{secEstPi}
\allowdisplaybreaks

In the proof of Theorem \ref{pi} we estimate the function $\pi(x;q,a)$ using the function $\psi(x;q,a)$. In order to do the estimate, we also compute certain integrals. The next lemma is used to obtain the estimates.

\begin{lemma}
\label{lemmaEstS}
Let $q \geq 3$ and $a$ be integers such that $\gcd(q,a)=1$ and $x \geq q$ be a real number. Let us write $\psi(x;q,a)=\frac{x}{\varphi(q)}+S(x)$ where the term $S(x)$ can be estimated by Theorem \ref{psi}. Then
\begin{align*}
&\frac{S(x)}{\log x}+\int_2^x\frac{S(t)}{t\log^2 t}dt < \left(\frac{1}{8\pi \varphi(q)}+\frac{1}{6\pi}\right)\sqrt{x}\log x  \\
& \quad+\left(0.184\log q+8.396\right)\sqrt{x}+\left(6.05\log q+157.318\right)\frac{\sqrt{x}}{\log x} \\
&\quad+ \left(5.048\log^2 q+152.085\log{q}+1719.86\right)\frac{\sqrt{x}}{\log^2 x} \\
&\quad+\left(0.184\log q+8.250\right)x^{1/4}\log\log{\sqrt{x}}+\left(5.254\log^2 q+121.765\log{q} \right. \\
&\quad\left.+937.202 \vphantom{5.254\log^2}\right)x^{1/4}-\left(11.364\log q+281.636\right)\frac{x^{1/4}}{\log x}\\
&\quad-\left(10.096\log^2 q+261.658\log{q}+2445.176\right)\frac{x^{1/4}}{\log^2 x}  \\
&\quad+\left(80.768\log^2 q+1753.168\log{q}+11605.056\right)\frac{x^{1/4}}{\log^3 x} \left(x^{1/4}-1\right) \\
&\quad+\left(2.015\log q+0.5\right)\log\log x+2.754\log{q}+0.534+\frac{R_1(q)}{\log{2}},
\end{align*}
where $R_1(q)$ is defined as in Theorem \ref{psi}. 
\end{lemma}
\begin{proof}
By Theorem \ref{psi} we have
\begin{equation}
\label{EstS}
\begin{aligned}
\left|S(x)\right| &\leq \left(\frac{1}{8\pi \varphi(q)}+\frac{1}{6\pi}\right)\sqrt{x}\log^2 x+ \left(0.184\log q+8.250\right)\sqrt{x}\log x\\
&\quad+(5.314\log q+124.318)\sqrt{x}+(5.048\log^2q+109.573\log q+725.316)\frac{\sqrt{x}}{\log {x}}\\ 
 &\quad+\left(2.015\log q+0.5\right)\log x +\left|R_1(q)\right|.
\end{aligned}
\end{equation}
First we notice that the term $R_1(q)$ is constant with respect to $x$. Since $\int_2^x \frac{dt}{t\log^2 t}=\frac{1}{\log 2}-\frac{1}{\log x}$ and $R_1(q)>0$ for all numbers $q$, 
we have 
\begin{equation}
\label{estForR}
\frac{\left|R_1(q)\right|}{\log{x}} +\int_2^x\frac{\left|R_1(q)\right|}{t\log^2 t}dt=\frac{R_1(q)}{\log{2}}.
\end{equation}
Thus it is sufficient to estimate the term $\int_2^x\frac{S(t)}{t\log^2 t}dt$ without the $R$ terms.

We can bound the first term coming from the term $\int_2^x\frac{S(t)}{t\log^2 t}dt$ simply by integrating:
\begin{equation}
\label{EstSFirst}
 \left(\frac{1}{8\pi \varphi(q)}+\frac{1}{6\pi}\right)\int_2^x\frac{dt}{\sqrt{t}}\leq  \left(\frac{1}{4\pi \varphi(q)}+\frac{1}{3\pi}\right)\left(\sqrt{x}-\sqrt{2}\right).
\end{equation}

Before considering the second term, let us note that for any nonnegative function $f(t)$, the inequality 
\begin{equation*}
\int_2^{\sqrt{x}} f(t)\sqrt{t} dt \leq x^{1/4}\int_2^{\sqrt{x}} f(t) dt
\end{equation*} 
clearly holds when $\sqrt{x} \geq 2$. It holds also when $\sqrt{x}<2$, since $\int_2^{\sqrt{x}}=-\int_{\sqrt{x}}^2$ and we also have $x^{1/4} \leq \sqrt{t}$ in this case. This fact is applied when the estimates for the second, third and fourth terms are derived.

In the second term, we want to treat separately the small and the large values of $t$:
\begin{equation}
\label{EstSSecond}
\begin{aligned}
& \left(0.184\log q+8.250\right)\int_2^x\sqrt{t}\log t\cdot \frac{1}{t\log^2t}dt\\ 
&\quad\leq  \left(0.184\log q+8.250\right)\left(x^{1/4}\int_2^{\sqrt{x}}\frac{1}{t\log t}dt+\frac{2}{\log x}\int_{\sqrt{x}}^x \frac{1}{\sqrt{t}}dt\right)\\ 
&\quad=x^{1/4} \left(0.184\log q+8.250\right)\left(\log\log{\sqrt{x}}-\log\log 2\right) \\
&\quad\quad+\frac{2}{\log x} \left(0.184\log q+8.250\right)\left(2\sqrt{x}-2x^{1/4}\right).
\end{aligned}
\end{equation}

The third term coming from the term $\int_2^x\frac{S(t)}{t\log^2 t}dt$ can also be estimated by dividing it to two parts:
\begin{equation}
\label{EstSThird}
\begin{aligned}
& (5.314\log q+124.318)\int_2^x\frac{\sqrt{t}}{t\log^2 t}dt \\
& \quad \leq (5.314\log q+124.318)\left(\int_2^{\sqrt{x}} \frac{\sqrt{t}}{t\log^2 t} dt+\int_{\sqrt{x}}^x \frac{dt}{\sqrt{t}\log^2{t}} \right) \\
&\quad\leq (5.314\log q+124.318)\left(x^{1/4}\int_2^{\sqrt{x}} \frac{dt}{t\log^2 t}+\frac{1}{\log^2 (\sqrt{x})}\int_{\sqrt{x}}^x \frac{dt}{\sqrt{t}}\right)\\ 
&\quad\leq (5.314\log q+124.318)\left(\frac{8\sqrt{x}}{\log^2 x}+\frac{x^{1/4}}{\log 2}-\frac{2x^{1/4}}{\log x}-\frac{8x^{1/4}}{\log^2x}\right).
\end{aligned}
\end{equation}

Let us move on to the fourth term coming from the term $\int_2^x\frac{S(t)}{t\log^2 t}dt$. We will forget the constant coefficient in front of the term for a while. Again, dividing the integral to two parts, we get
\begin{equation}
\label{EstSFourth}
\begin{aligned}
\int_2^x \frac{\sqrt{t}}{t\log^3 t}dt&\leq x^{1/4}\int_2^{\sqrt{x}} \frac{dt}{t\log^3t}+\frac{1}{\log^3\sqrt{x}}\int_{\sqrt{x}}^x\frac{dt}{\sqrt{t}}\\ 
&=x^{1/4}\left(\frac{1}{2\log^2 2}-\frac{4}{2\log^2x}\right)+\frac{8}{\log^3x}\left(2\sqrt{x}-2x^{1/4}\right)\\ 
&=\frac{16\sqrt{x}}{\log^3x}+\frac{x^{1/4}}{2\log^22}-\frac{2x^{1/4}}{\log^2x}-\frac{16x^{1/4}}{\log^3x}.
\end{aligned}
\end{equation}

Simply integrating the fifth term coming from the term $\int_2^x\frac{S(t)}{t\log^2 t}dt$ is
\begin{equation}
\label{EstSFifth}
\begin{aligned}
\left(2.015\log q+0.5\right)\int_2^x \frac{\log t}{t\log^2 t}dt&=\left(2.015\log q+0.5\right)\int_2^x \frac{dt}{t\log t}dt \\
&=\left(2.015\log q+0.5\right)\left(\log\log x-\log\log 2\right).
\end{aligned}
\end{equation}

The claim follows from estimates \eqref{estForR}, \eqref{EstS}, \eqref{EstSFirst}, \eqref{EstSSecond}, \eqref{EstSThird}, \eqref{EstSFourth}  and \eqref{EstSFifth} and since $\varphi(q) \geq 2$.
\end{proof}

As we see in Lemma \ref{lemmaEstS}, the estimates which appear in the proof of Theorem \ref{pi} can be quite long. Thus we would like to simplify them. We estimate the terms which appear in the estimate for the function $\pi(x;q,a)$ and which are asymptotically at most of size $O(\log\log{x})$ or have a negative sign. Notice that the numerical values appearing in this lemma, come both from the estimates proved in Section \ref{secInequPi} and from terms derived in the proof of Theorem \ref{pi}.
\begin{lemma}
\label{lemmaPiLowerWhole}
Let $q \geq 3$ be an integer, $x \geq q$ be a real number and $L(s, \chi)$ a Dirichlet $L$-function with conductor $q$ and the function $R_1(q)$ defined as in Theorem \ref{psi}. Then we have
\begin{align*}
&-\left(11.364\log q+284.488\right)\frac{x^{1/4}}{\log x}-\left(10.096\log^2 q+261.658\log{q}+2456.585\right)\frac{x^{1/4}}{\log^2x} \\ 
& \quad-\left(80.768\log^2 q+1753.168\log{q}+11605.056\right)\frac{x^{1/4}}{\log^3x} \\
& \quad +\left(2.015\log q+0.5\right)\log\log x+2.754\log{q}+1.455+\frac{R_1(q)}{\log{2}} \\
& \quad <-237.934.
\end{align*}
\end{lemma}

\begin{proof}
The first three lines of the claim are already estimated by Lemma \ref{lemmaPiLowerSub}. Thus we just combine the results and prove the claim dividing it to different cases depending on the value of $q$.

By estimate \eqref{kanselloiva} and the definition of the function $R_1(q)$ for $3\leq q<4\cdot 10^5$ the left hand side of the claim is
\begin{equation}
\label{negative1}
\begin{aligned}
&< -\left(\left(2.015\log{q}+0.5\right)\log\log x +2.104\log^2q+55.018\log q+611.027\right) \\
& \quad+\left(2.015\log q+0.5\right)\log\log x+2.754\log{q}+1.455 \\
&\quad+\frac{0.014\sqrt{4\cdot 10^5}\log{q}+0.034\sqrt{4\cdot 10^5} +3.679\log q+263.886}{\log{2}} \\ 
& \quad< -2.104\log^2q-34.182\log q-197.842<-237.934.
\end{aligned}
\end{equation}

Similarly for $4\cdot 10^5\leq q< 10^{29}$ we get that the left hand side of the claim is
\begin{equation}
\label{negative2}
\begin{aligned}
&< -\left(\left(2.015\log{q}+0.5\right)\log\log x +2.104\log^2q+55.018\log q+611.027\right) \\
& \quad+\left(2.015\log q+0.5\right)\log\log x+2.754\log{q}+1.455 \\
&\quad+\frac{1.858\log^2{q}+3.679\log q+104.626}{\log{2}} \\ 
& \quad< 0.577\log^2q-46.956\log q-458.628<-968.316.
\end{aligned}
\end{equation}

Finally, we only have left the large values of $q$ i.e. $q \geq  10^{29}$. We use estimate \eqref{kanselloiva2} instead of estimate \eqref{kanselloiva}. Keeping in mind that $x \ge q$, the
left hand side of the claim is
\begin{equation}
\label{negative3}
\begin{aligned}
&< -\left(\left(\frac{0.297}{\log 2}\log^2q+2.015\log{q}+0.5\right)\log\log x\right. \\
&\quad\left. +45086.567\log^2q+2.8\cdot10^6\log q+8.5\cdot10^7 \vphantom{\frac{0.297}{\log 2}}\right)+\left(2.015\log q+0.5\right)\log\log x \\
&\quad+2.754\log{q}+1.455+\frac{(0.297\log\log q+0.603)\log^2{q}+3.679\log q+104.626}{\log{2}} \\ 
& \quad< -45085.697\log^2q-2.7\cdot10^6\log q-8.4\cdot10^7<-10^7.
\end{aligned}
\end{equation}

Comparing estimates \eqref{negative1}, \eqref{negative2} and \eqref{negative3} we obtain the wanted result.
\end{proof}

\subsection{Proof of Theorem \ref{pi} and Corollary \ref{piCorollary}}

Next we combine the results proved in Sections \ref{secProofPsi} and \ref{secEstPi} and prove our main result, i.e. estimate the function $\pi(x;q,a)$.
\begin{proof}[Proof of Theorem \ref{pi}]
The main idea of the proof is to estimate the function $\pi(x;q,a)$ by the function 
\begin{equation*}
\theta(x;q,a)=\sum_{\substack{p\leq x \\ p\equiv a\bmod q }}\log{p},
\end{equation*}
where $\gcd(q,a)=1$. Further, the function $\theta(x;q,a)$ can be estimated by the function $\psi(x;q,a)=\sum_{n\leq x, n\equiv a\bmod q}\Lambda(n)$ and thus we can apply Theorem \ref{psi}. After careful computations, the result will be obtained.

First, we obtain $\pi(x;q,a)$ from $\theta(x;q,a)$ using partial summation. We have
\begin{equation*}
\pi(x;q,a)=\sum_{\substack{p\leq x\\ p\equiv a\bmod q}}\frac{\log{p}}{\log{p}}=\frac{\theta(x;q,a)}{\log x}+\int_2^x\frac{\theta(t;q,a)}{t\log^2t}dt
\end{equation*}
which can be written in the form
\begin{equation}
\label{piEstpsi}
\frac{\psi(x;q,a)}{\log x}+\int_2^x\frac{\psi(t;q,a)}{t\log^2t}dt+\frac{\theta(x;q,a)-\psi(x;q,a)}{\log x}+\int_2^x\frac{\theta(t;q,a)-\psi(t;q,a)}{t\log^2t}dt.
\end{equation}

Let us start with the last two terms of the previous formula. The difference $\psi(x;q,a)-\theta(x;q,a)$ can be estimated by \cite{RosserSchoenfeldApproximate} Theorem 13:
\begin{multline*}
0<\psi(x;q,a)-\theta(x;q,a)=\sum_{\substack{p^{\alpha}\leq x \\ p^{\alpha} \equiv a\pmod q\\ \alpha \geq 2}} \Lambda\left(p^{\alpha}\right)
\leq \sum_{\substack{p^{\alpha}\leq x \\\alpha \geq 2}} \Lambda\left(p^{\alpha}\right) \\
=\psi(x)-\theta(x)<1.4262\sqrt{x}.
\end{multline*}
Thus, let us start with the error terms:
\begin{equation}
\label{thetaEstpsi}
\left|\frac{\theta(x;q,a)-\psi(x;q,a)}{\log x}\right|<1.4262\frac{\sqrt{x}}{\log x}
\end{equation}
and, similarly as in inequality \eqref{EstSThird}, we can estimate
\begin{equation}
\label{thetaIntEstpsi}
\begin{aligned}
&\left|\int_2^x\frac{\theta(t;q,a)-\psi(t;q,a)}{t\log^2t}dt\right| \leq 1.4262\int_2^{\sqrt{x}} \frac{\sqrt{t}}{t\log^2 t} dt+1.4262\int_{\sqrt{x}}^x \frac{dt}{\sqrt{t}\log^2{t}}\\
& \quad\quad=11.4096\frac{\sqrt{x}}{\log^2 x}+1.4262\frac{x^{1/4}}{\log 2}-2.8524\frac{x^{1/4}}{\log x}-11.4096\frac{x^{1/4}}{\log^2 x}.
\end{aligned}
\end{equation}
Next we move on to estimate the first two terms on formula \eqref{piEstpsi}. 

Let us now write $\psi(x;q,a)=\frac{x}{\varphi(q)}+S(x)$ to separate the main term from the rest of the terms. (Remember that the term $S(x)$ can be estimated by Theorem \ref{psi}.) We have
\begin{equation*}
\frac{\psi(x;q,a)}{\log x}+\int_2^x\frac{\psi(t;q,a)}{t\log^2t}dt=\frac{x}{\varphi(q)\log x}+\frac{S(x)}{\log x}+\frac{1}{\varphi(q)}\int_2^x\frac{dt}{\log^2t}+\int_2^x\frac{S(t)}{t\log^2 t}dt.
\end{equation*}
Since
\begin{equation*}
\frac{1}{\varphi(q)}\int_2^x\frac{dt}{\log^2t}=\frac{1}{\varphi(q)}\left(\mathrm{li}(x)-\frac{x}{\log x}-\mathrm{li}(2)+\frac{2}{\log 2}\right),
\end{equation*}
we have 
\begin{equation}
\label{estPiPsiMain}
\frac{\psi(x;q,a)}{\log x}+\int_2^x\frac{\psi(t;q,a)}{t\log^2t}dt=\frac{\mathrm{li}(x)}{\varphi(q)}+\frac{S(x)}{\log x}+\int_2^x\frac{S(t)}{t\log^2 t}dt+\frac{1}{\varphi(q)}\left(-\mathrm{li}(2)+\frac{2}{\log 2}\right).
\end{equation}
This can be estimated by Lemma \ref{lemmaEstS}.

It is time to put everything together. By estimates \eqref{piEstpsi}, \eqref{thetaEstpsi}, \eqref{thetaIntEstpsi} and \eqref{estPiPsiMain} and Lemma \ref{lemmaEstS} we have 
\begin{align*}
&\left|\pi(x;q,a)-\frac{\mathrm{li}(x)}{\varphi(q)}\right| \leq 1.4262\frac{\sqrt{x}}{\log x}+11.4096\frac{\sqrt{x}}{\log^2 x}+1.4262\frac{x^{1/4}}{\log 2}-2.8524\frac{x^{1/4}}{\log x}\\
& \quad-11.4096\frac{x^{1/4}}{\log^2 x} +\frac{1}{\varphi(q)}\left|-\mathrm{li}(2)+\frac{2}{\log 2}\right|+ \left(\frac{1}{8\pi \varphi(q)}+\frac{1}{6\pi}\right)\sqrt{x}\log x \\
& \quad+\left(0.184\log q+8.396\right)\sqrt{x}+\left(6.05\log q+157.318\right)\frac{\sqrt{x}}{\log x}\\
&\quad+ \left(5.048\log^2 q+152.085\log{q}+1719.86\right)\frac{\sqrt{x}}{\log^2 x} \\
&\quad+\left(0.184\log q+8.250\right)x^{1/4}\log\log{\sqrt{x}}\\
&\quad+\left(5.254\log^2 q+121.765\log{q}+937.202\right)x^{1/4}-\left(11.364\log q+281.636\right)\frac{x^{1/4}}{\log x} \\
&\quad-\left(10.096\log^2 q+261.658\log{q}+2445.176\right)\frac{x^{1/4}}{\log^2 x}  \\
&\quad+\left(80.768\log^2 q+1753.168\log{q}+11605.056\right)\frac{x^{1/4}}{\log^3 x} \left(x^{1/4}-1\right) \\
&\quad+\left(2.015\log q+0.5\right)\log\log x+2.754\log{q}+0.534+\frac{R_1(q)}{\log{2}}.
\end{align*}
The right hand side of the inequality can be simplified to 
\begin{align*}
&< \left(\frac{1}{8\pi \varphi(q)}+\frac{1}{6\pi}\right)\sqrt{x}\log x+\left(0.184\log q+8.396\right)\sqrt{x}+\left(6.05\log q \right. \\
&\quad\left.+158.745\right)\frac{\sqrt{x}}{\log x}+ \left(5.048\log^2 q+152.085\log{q}+1731.270\right)\frac{\sqrt{x}}{\log^2 x} \\
&\quad+\left(0.184\log q+8.250\right)x^{1/4}\log\log{\sqrt{x}}\\
&\quad+\left(5.254\log^2 q+121.765\log{q}+939.260\right)x^{1/4}-\left(11.364\log q+284.488\right)\frac{x^{1/4}}{\log x} \\
&\quad-\left(10.096\log^2 q+261.658\log{q}+2456.585\right)\frac{x^{1/4}}{\log^2 x}  \\
&\quad+\left(80.768\log^2 q+1753.168\log{q}+11605.056\right)\frac{x^{1/4}}{\log^3 x} \left(x^{1/4}-1\right) \\
&\quad+\left(2.015\log q+0.5\right)\log\log x+2.754\log{q}+1.455+\frac{R_1(q)}{\log{2}}.
\end{align*}

The estimate is still quite long and we would like to simplify it even more. The terms which are asymptotically at most of size $O(\log\log{x})$ and the terms which have negative sign can be estimated by Lemmas \ref{lemmaPiLowerSub} and \ref{lemmaPiLowerWhole}. Thus the previous one can be estimated with 
\begin{align*}
&< \left(\frac{1}{8\pi \varphi(q)}+\frac{1}{6\pi}\right)\sqrt{x}\log x+\left(0.184\log q+8.396\right)\sqrt{x}+\left(6.05\log q \right. \\
&\quad+\left. +158.745\right)\frac{\sqrt{x}}{\log x}+\left(5.048\log^2 q+152.085\log{q}+1731.270\right)\frac{\sqrt{x}}{\log^2 x}\\
&\quad+\left(0.184\log q+8.250\right)x^{1/4}\log\log{x}\\
&\quad+\left(5.254\log^2 q+121.765\log{q}+939.260\right)x^{1/4} \\
&\quad+\left(80.768\log^2 q+1753.168\log{q}+11605.056\right)\frac{\sqrt{x}}{\log^3 x}-237.934,
\end{align*}
which proves the claim.
\end{proof}

At the end of this article, we prove Corollary \ref{piCorollary}:
\begin{proof}[Proof of Corollary \ref{piCorollary}]
Let us simplify the expression in Theorem \ref{pi}. First we use the fact that $x \geq q$ and estimate
\begin{equation}
\label{eqAlmost}
\begin{aligned}
&\left|\pi(x;q,a)-\frac{\mathrm{li}(x)}{\varphi(q)}\right|< \left(\frac{1}{8\pi \varphi(q)}+\frac{1}{6\pi}\right)\sqrt{x}\log x+\left(0.184\log q+8.396+6.05 \right. \\
&\quad\left.+5.048\right)\sqrt{x} +\left(158.745+152.085+80.768\right)\frac{\sqrt{x}}{\log x} \\
&\quad+ \left(1731.270+1753.168\right)\frac{\sqrt{x}}{\log^2 x}+\left(0.184\log x+8.250\right)x^{1/4}\log\log{x} \\
&\quad+\left(5.254\log^2 x+121.765\log{x}+939.260\right)x^{1/4}+11605.056\frac{\sqrt{x}}{\log^3 x}-237.934,
\end{aligned}
\end{equation}

Note that for $x \geq 3$ we have $\frac{\sqrt{x}}{\log{x}}\leq 0.911\sqrt{x}$, $\frac{\sqrt{x}}{\log^2{x}}\leq 0.829\sqrt{x}$, $\frac{\sqrt{x}}{\log^3{x}}\leq 0.755\sqrt{x}$ and $x^{1/4}<0.760\sqrt{x}$. Furthermore, by Lemma \ref{estForNegative} we have $x^{1/4}\log{x} \log\log x<\frac{\sqrt{x}}{0.416}$ and $x^{1/4}\log{\log{x}}<0.524\sqrt{x}$. Further by the proof of Lemma \ref{lemmaPiLowerSub} we have $x^{1/4}\log^2 x\le \frac{64}{e^2}\sqrt{x}$ and $x^{1/4}\log x\le \frac{4}{e}\sqrt{x}$. Thus formula \eqref{eqAlmost} yields the estimate
\begin{multline*}
\left|\pi(x;q,a)-\frac{\mathrm{li}(x)}{\varphi(q)}\right| \leq \left(\frac{1}{8\pi \varphi(q)}+\frac{1}{6\pi}\right)\sqrt{x}\log x \\
+\left(0.184\log q+12969.946\right)\sqrt{x}-237.934.
\end{multline*}
\end{proof}

\bibliographystyle{amsplain}

\end{document}